\documentclass[11pt]{article}
\usepackage[left=1.5cm, right=1.5cm, top=1.7cm, bottom=2cm]{geometry}
\usepackage{amsmath}
\usepackage{amsthm}
\usepackage{amssymb}
\usepackage[T1]{fontenc}
\usepackage{enumerate} 
\usepackage{mathtools}
\usepackage{bm}
\usepackage[hidelinks]{hyperref}
\usepackage{cleveref}
\usepackage{lineno}
\usepackage{mathrsfs}
\usepackage{algorithm}
\usepackage{algpseudocode}
\usepackage{cite}

\normalfont

\DeclareMathOperator{\Div}{div}
\let\d\relax
\DeclareMathOperator{\d}{d\!}

\allowdisplaybreaks

\theoremstyle{plain}
\newtheorem{theorem}{Theorem}[section]
\newtheorem{lemma}[theorem]{Lemma}
\newtheorem{proposition}[theorem]{Proposition}
\newtheorem{corollary}[theorem]{Corollary}

\newtheorem{assumption}[theorem]{Assumption}

\theoremstyle{definition}
\newtheorem{remark}[theorem]{Remark}

\crefname{lemma}{Lemma}{Lemmata}
\crefname{theorem}{Theorem}{Theorems}
\crefname{proposition}{Proposition}{Propositions}
\crefname{corollary}{Corollary}{Corollarys}
\crefname{algorithm}{Algorithm}{Algorithms}
\crefname{assumption}{Assumption}{Assumptions}
\crefname{definition}{Definition}{Definitions}
\crefname{remark}{Remark}{Remarks}
\crefname{proof}{proof}{proofs}
\crefname{equation}{}{}

\title{
 \bf Analysis of the SQP Method for Hyperbolic PDE-Constrained Optimization in Acoustic Full Waveform Inversion\footnote{This work is supported by the DFG research grants YO159/4-1 and YO159/5-1.}
 }

\author{
 Luis Ammann\footnote{University of Duisburg-Essen, Fakult\"at f\"ur Mathematik, Thea-Leymann-Str. 9, D-45127 Essen, Germany. 
 \newline Email: luis.ammann@uni-due.de, irwin.yousept@uni-due.de}
 \and
 Irwin Yousept $^\dag$\footnote{Corresponding author}
 }

 \def\keywords{
 {\footnotesize Keywords: SQP methods, full waveform inversion, hyperbolic PDE-constrained optimization, well-posedness, R-superlinear convergence}
 }

\begin{document}
\date{}
\maketitle
\allowdisplaybreaks

\begin{abstract}
	In this paper, the SQP method applied to a hyperbolic PDE-constrained optimization problem is considered. The model arises from the acoustic full waveform inversion in the time domain. The analysis is mainly challenging due to the involved hyperbolicity and second-order bilinear structure. This notorious character leads to an undesired effect of loss of regularity in the SQP method, calling for a substantial extension of developed parabolic techniques. We propose and analyze a novel strategy for the well-posedness and convergence analysis based on the use of a smooth-in-time initial condition, a tailored self-mapping operator, and a two-step estimation process along with Stampacchia's method for second-order wave equations. Our final theoretical result is the R-superlinear convergence of the SQP method. 
\end{abstract}

\keywords

\section{Introduction}
\label{section:introduction}

This paper analyzes the Sequential Quadratic Programming (SQP) method for a class of hyperbolic optimal control problems with applications in acoustic full waveform inversion. To describe the model problem, let $\Omega \subset \mathbb R^N$ ($N\geq 2$) be a bounded domain with a Lipschitz boundary $\partial \Omega$ and $I \coloneqq [0, T]$ be a finite time interval. The boundary of $\Omega$ is is given by $\partial\Omega = \Gamma_D\cup\Gamma_N$ with a closed subset $\Gamma_D\subsetneq\partial\Omega$ satisfying $|\Gamma_D| \neq 0$ and $\Gamma_N = \partial\Omega\setminus\Gamma_D$. Introducing the square slowness in $\Omega$ by $\nu\colon \Omega \to \mathbb R$, the propagation of the acoustic pressure in $\Omega$ can be mathematically described by the solution $p\colon I\times\Omega\to\mathbb R$ to the following damped acoustic wave equation:
\begin{equation}
	\label{system: state}
	\left\{\begin{aligned}
		&\nu\partial_{tt}p - \Delta p + \eta \partial_t p = f 
		&&\text{in }I\times \Omega
		\\
		&\partial_n p = 0 
		&&\text{in }I\times\Gamma_N
		\\
		&p = 0 
		&&\text{in }I\times\Gamma_D
		\\
		&(p,\partial_t p)(0) = (0,0) 
		&&\text{in }\Omega.
	\end{aligned}\right.
\end{equation}
Here, $\eta\colon\Omega\to\mathbb R$ is a given damping coefficient, and $f\colon\Omega\to\mathbb R$ is a given source term. Full waveform inversion (FWI) is a famous method for reconstructing the square slowness $\nu$. A suitable FWI formulation is given by the PDE-constrained optimization problem
\begin{equation} 
	\label{P}\tag{P}
	\left\{\begin{aligned}
		& \inf \mathcal{J}(\nu, p) \coloneqq\frac{1}{2} \sum_{i=1}^m \int_{0}^{T}\int_{\Omega}a_i(p-p^{ob}_i)^2\d x\d t+\frac{\lambda}{2}\|\nu\|_{L^2(\Omega)}^2
		\\
		&\text{ s.t. } \cref{system: state} \text{ and } \nu\in\mathcal V_{ad}\coloneqq\{\nu\in L^2(\Omega): \nu_- (x)\leq \nu(x)\leq \nu_+(x)\text{ for a.e. }x\in\Omega\}
	\end{aligned}\right.
\end{equation}
for some given observation data $p^{ob}_{i}\colon I\times\Omega\to\mathbb R$ recorded at receivers modeled through the weight functions $a_i\colon I\times\Omega\to\mathbb R$. Moreover, the constant $\lambda >0$ denotes the Tikhonov regularization parameter, and $\nu_-\colon\Omega\to\mathbb R$ (resp. $\nu_+\colon\Omega\to\mathbb R$) denotes the lower (resp. upper) bound for the square slowness $\nu$. For a more detailed discussion on the appearing quantities, particularly their physical explanation, we refer to our previous work \cite{ammann23}. For an extensive overview of FWI, we refer to Fichtner \cite{fichtner11}, Virieux and Operto \cite{operto09}, and the references therin.

The SQP method is a celebrated technique in finite and infinite dimensional optimization, particularly in the context of optimal control problems. We refer to the earlier contributions by Alt \cite{alt90,alt98} and Alt, Sontag, and Tr\"oltzsch \cite{alt96} for SQP methods of optimization problems with ODE or integral equations constraints. From among many other related works in the context of PDE-constrained problems, we mention the contributions by Ito and Kunisch \cite{ito96,ito96'}, Tr\"oltzsch et al. \cite{goldberg98,heinkenschloss99,troeltzsch99,troeltzsch00,troeltzsch01}, Heinkenschloss \cite{heinkenschloss98}, Hinterm\"uller and Hinze \cite{hintermueller02,hintermueller06}, Volkwein \cite{volkwein03}, Wachsmuth \cite{wachsmuth07}, Griesse et al. \cite{griesse08,griesse10}, Hinze and Kunisch \cite{hinze01}, and Hoppe and Neitzel \cite{hoppe21}. Even though the investigations of SQP methods are highly problem-specified, they mainly follow the same methodology: Reformulation of the SQP method as a generalized Newton method and exploitation of Robinson's concept of \textit{strong regularity} \cite{robinson80}. This unified ansatz leads to well-posedness and quadratic convergence of SQP iterations. Eventually, one verifies the strong regularity condition using suitable second-order sufficient optimality conditions.

Notice that the works mentioned above only focus on elliptic and parabolic PDEs. Our paper seems to be the first contribution toward analyzing SQP methods in hyperbolic PDE-constrained optimization. For our model problem \cref{P}, this results in a challenging task due to the underlying hyperbolicity and the second-order bilinear structure $\nu\partial_t^2 p$. This character leads to an undesired effect of \textit{loss of regularity} in the SQP method (see \cref{alg:sqp}) causing two substantial difficulties (see \cref{remark: GE}): (i) In general, \cref{alg:sqp} is only executable for a limited number of iterations, i.e., the well-definedness of \cref{alg:sqp} may fail. (ii) The ansatz through the notion of strong regularity as done in the parabolic case (cf. \cite{troeltzsch99}) cannot be directly transferred to our case and requires a substantial extension.

This paper develops a novel strategy for analyzing \cref{alg:sqp} and consists of three primary steps. First of all, we propose the use of a smooth-in-time initial guess for the state $p_0$ and the adjoint state $q_0$ satisfying $\partial_t^l p_0(0) = \partial_t^l q_0(T) = 0$ for all $l\in\mathbb N$ (\cref{assumption3}). Under this regularity condition, we manage to prove the well-definedness of \cref{alg:sqp} (see \cref{proposition: well-definedness}). As the second step, for every given SQP iteration $(\nu_k, p_k, q_k)$, we construct a tailored self-mapping operator \cref{restriction} using certain solution operators $S_k$ and $T_k$, respectively, for the perturbed optimality conditions \cref{system: Sk} and the PDE-systems in the SQP iteration \cref{iteration}. Based on perturbation analysis (see \cref{theorem:lipschitzproperty}) using Stampacchia's method (see \cref{lemma: stampaccia}), it turns out that the contraction principle can be applied to the operator \cref{restriction} (see \cref{proposition contraction}). The resulting fixed point $\nu_{k+1}$ is exactly the control component of the solution to the SQP iteration \cref{iteration} (see \cref{proposition uniqueness}). The final step comprises a \textit{two-step} estimation process: We estimate $\|\nu_{k+1} - \overline\nu\|_{L^2(\Omega)}$ by the total error of the previous step $\|\nu_k - \overline\nu\|_{L^2(\Omega)}$, $\|p_k - \overline p\|_{L^2(I, L^2(\Omega))}$, and $\|q_k - \overline q\|_{L^2(I, L^2(\Omega))}$. Then, the error in the state $\|p_k - \overline p\|_{L^2(I, L^2(\Omega))}$ and adjoint state $\|q_k - \overline q\|_{L^2(I, L^2(\Omega))}$ are estimated towards $\|\nu_{k-1} - \overline\nu\|_{L^2(\Omega)}$. This process results in the quadratic two-step estimation $\|\nu_{k+1} - \overline\nu\|_{L^2(\Omega)}\leq \delta (\|\overline\nu - \nu_k\|_{L^2(\Omega)} + \|\overline\nu - \nu_{k-1}\|_{L^2(\Omega)})^2$ which eventually allows us to prove our main result on the R-superlinear convergence of \cref{alg:sqp} (see \cref{theorem: convergence}).

\section{PDE-constrained optimization for FWI}
\label{section optimization}

We denote the space of all equivalence classes of measurable and Lebesgue square integrable $\mathbb R$-valued functions by $L^2(\Omega)$. Furthermore, let 
\begin{align*}
	H^1_D(\Omega)\coloneqq\{& v \in H^1(\Omega)\colon\tau v =0 \text{ on } \Gamma_D\}
	\\
	\bm H_N(\Div,\Omega)\coloneqq\{&\bm{u}\in \bm H(\Div, \Omega)\colon(\Div\bm{u},\phi)_{L^2(\Omega)}= - (\bm{u},\nabla\phi)_{\bm{L}^2(\Omega)}\,\forall \phi\in H^1_D(\Omega)\},
	\\
	D(\Delta_{D,N}) \coloneqq \{&v\in H^1_D(\Omega): \nabla v \in \bm H_N(\Div, \Omega)\}
\end{align*} 
where the gradient and divergence are understood in the weak sense and $\tau\colon H^1(\Omega)\to L^2(\partial\Omega)$ denotes the trace operator.

\begin{assumption}\label{assumption1}
	Let $f\in W^{1,1}(I,L^2(\Omega))$ and $\eta \in L^\infty(\Omega)$ satisfying $\eta(x) \ge 0$ for a.e. $x\in\Omega$. Let $m\in\mathbb N$. For all $i=1,\ldots, m$, we suppose that $p^{ob}_{i} \in L^2(I,L^2(\Omega))$ and $a_i \in L^\infty(I\times \Omega)$ satisfying $a_i(t,x) \ge 0$ for a.e. $(t,x) \in I \times \Omega$. Furthermore, let $\nu_{\min}>0$, $\nu_{\max}>0$ and $\nu_-,\nu_+ \in L^\infty(\Omega)$ satisfy $\nu_{\min} \leq\nu_- (x) \leq\nu_+ (x)\leq\nu_{\max}$ for a.e. $x\in\Omega$.
\end{assumption}

The following lemma provides regularity conditions for the solutions of the second-order state equation. The result follows immediately from \cite[Lemma 2.2]{ammann23} (cf. also\cite[Theorem 5 on p. 410]{Evans2010}).

\begin{lemma}
	\label{lemma: abstract}
	Let \cref{assumption1} hold and let $\nu\in \mathcal V_{ad}$. Further, let $g\in W^{k,1}(I,L^2(\Omega))$ for some $k\geq 1$ and $\partial_t^lg(0)= 0$ for $l=0,\ldots, k-1$. Then, the unique solution $p$ to
	\begin{equation}\label{system: acoustic}
		\left\{\begin{aligned}
			&\nu\partial_t^2p - \Delta p + \eta \partial_t p = g 
			&&\text{in }I\times \Omega
			\\
			&\partial_n p = 0 
			&&\text{in }I\times\Gamma_N
			\\
			&p = 0
			&&\text{in }I\times\Gamma_D
			\\
			&(p,\partial_t p)(0) = (0,0) 
			&&\text{in }\Omega 
		\end{aligned}\right.
	\end{equation}
	satisfies $p\in C^{k+1}(I,L^2(\Omega))\cap C^k(I,H^1_D(\Omega))\cap C^{k-1}(I,D(\Delta_{D,N}))$ and it holds that
	\begin{align*}
		\|p(t)\|_{L^2(\Omega)}
		\leq c\|G\|_{L^1(I,L^2(\Omega))} \quad\text{and}\quad \|\partial_t^lp(t)\|_{L^2(\Omega)}
		\leq c\|\partial_t^{l-1} g\|_{L^1(I,L^2(\Omega))}
		\quad\forall t\in I\,\,\forall l = 1,\dots,k+1
	\end{align*}
	with $c\coloneqq \nu_{\min}^{-1}\frac{\max\left\{\sqrt{\nu_{\max}},1\right\}}{\min\{\sqrt{\nu_{\min}},1\}}$ and $G(t) \coloneqq \int_0^t g(s)\d s$ for all $t\in I$.
\end{lemma}

In our previous work \cite{ammann23}, we have shown an existence result for \cref{P} and derived its first-order necessary and second-order sufficient optimality conditions. The analysis in \cite{ammann23} makes use of an elliptic inner regularity result to obtain an inner boundedness of the corresponding state (see \cite[Lemma 4.3]{ammann23}). 
It turns out that the inner regularity ansatz can be improved by applying Stampacchia's method to the hyperbolic case that provides even the global $L^\infty(\Omega)$-boundedness. This approach allows us to avoid the structural assumption on the admissible set $\mathcal V_{ad}$ as in \cite[Assumption 4.1]{ammann23}.

\begin{lemma}[Global boundedness]
	\label{lemma: stampaccia}
	Let $N\leq 3$ and $g\in W^{k,1}(I,L^2(\Omega))$ for some $k\geq 1$ and $\partial_t^l g(0) = 0$ for $l=0,\ldots, k-1$. Then, the unique solution $p$ to \cref{system: acoustic} satisfies $p\in C^{k-1}(I,L^\infty(\Omega))$ and
	\begin{equation*}
		\|\partial_t^lp\|_{C(I,L^\infty(\Omega))}\leq \check c(\|\partial_t^lg\|_{L^1(I,L^2(\Omega))} + \|\partial_t^{l+1} g\|_{L^1(I,L^2(\Omega))})\quad\forall l=0,\ldots, k-1
	\end{equation*}
	for a constant $\check c>0$ independent of $p$, $g$, and $l$. As a consequence, if additionally $g\in H^k(I, L^2(\Omega))$, it holds that
	\begin{equation*}
		\|\partial_t^lp\|_{L^2(I,L^\infty(\Omega))}\leq \hat c(\|\partial_t^lg\|_{L^2(I,L^2(\Omega))} + \|\partial_t^{l+1} g\|_{L^2(I,L^2(\Omega))})\quad\forall l=0,\ldots, k-1
	\end{equation*}
	with $\hat c \coloneqq \check c T$.
\end{lemma}
\begin{proof} 
	Since $N\leq 3$, Stampacchia's method yields the existence of a constant $\tilde c>0$ such that 
	\begin{equation}\label{stampacchia estimate}
		\|y\|_{L^\infty(\Omega)} \leq \tilde c \|\Delta y\|_{L^2(\Omega)} \quad \forall y \in D(\Delta_{D,N}) \quad \Rightarrow \quad D(\Delta_{D,N}) \hookrightarrow L^\infty(\Omega). 
	\end{equation}
	This is obtained by classical arguments (cf. \cite[Theorem 4.5]{troeltzsch10}) along with the property $\max\{v,0\}\in H^1_D(\Omega)$ for all $v\in H^1_D(\Omega)$. On the other hand, since $g\in W^{k,1}(I, L^2(\Omega))$, \cref{lemma: abstract} implies that the unique solution $p$ to \cref{system: acoustic} satisfies $p\in C^{k+1}(I,L^2(\Omega))\cap C^k(I,H^1_D(\Omega))\cap C^{k-1}(I,D(\Delta_{D,N}))$. Thus, \cref{stampacchia estimate} implies that 
	$p\in C^{k-1}(I, L^\infty(\Omega))$ and 
	\begin{align*}
			\|p\|_{C(I, L^\infty(\Omega))} \leq \tilde c \|\Delta p &\|_{C(I, L^2(\Omega))} \underbrace{=}_{\cref{system: acoustic}} \tilde c \| g - \nu\partial_t^{2} p - \eta\partial_t p\|_{C(I,L^2(\Omega))}
			\\
			& \hspace{1.5cm} \underbrace{\leq}_{Lem. \ref{lemma: abstract}} \tilde c(\| g\|_{C(I,L^2(\Omega))} + c\nu_{\max}\|\partial_t g\|_{L^1(I,L^2(\Omega))} +c\|\eta\|_{L^\infty}\| g\|_{L^1(I,L^2(\Omega))})
			\\
			&\hspace{1.85 cm} \leq \hspace{0.34cm} \tilde c((1 + c\nu_{\max})\|\partial_t g\|_{L^1(I,L^2(\Omega))} +c\|\eta\|_{L^\infty}\|g\|_{L^1(I,L^2(\Omega))}),
		\end{align*}
	due to the fundamental theorem of calculus and $g(0) = 0$. Introducing $\check c\coloneqq\tilde c\max\{1 + c\nu_{\max},c\|\eta\|_{L^\infty}\}$, we arrive at 
	\begin{equation*}
		\|p\|_{C(I, L^\infty(\Omega))} \leq \check c(\|g\|_{L^1(I,L^2(\Omega))} + \|\partial_t g\|_{L^1(I,L^2(\Omega))}).	
	\end{equation*}
	Analogously, the estimate for $\partial_t^lp$ with $l\in \{1,\ldots, k-1\}$ follows from \cref{stampacchia estimate}, \cref{system: acoustic}, \cref{lemma: abstract} and $\partial_t^l g(0) = 0$. In conclusion, the first desired estimate is valid. The second desired estimate follows immediately from the first one.
\end{proof}

\begin{assumption}\label{assumption2}
	Let \cref{assumption1} hold and let $N\leq 3$, $f \in W^{6,1}(I,L^2(\Omega))$ with $\partial_t^lf(0) = 0$ for $l = 0,\ldots,5$. Furthermore, let $(\overline\nu, \overline p, \overline q)\in \mathcal V_{ad}\times C^2(I,L^2(\Omega))\cap C^1(I,H^1_D(\Omega))\cap C(I,D(\Delta_{D,N}))\times C^2(I,L^2(\Omega))\cap C^1(I,H^1_D(\Omega))\cap C(I,D(\Delta_{D,N}))$ satisfy the first-order necessary optimality condition
	\begin{equation}\label{first order optimality}
		\left\{\begin{aligned}
			&\overline\nu\partial_{t}^2\overline p - \Delta\overline p + \eta \partial_t\overline p = f 
			&&\text{in }I\times \Omega
			\\
			&\partial_n \overline p = 0 
			&&\text{in }I\times\Gamma_N
			\\
			&\overline p = 0
			&&\text{in }I\times\Gamma_D
			\\
			&(\overline p,\partial_t\overline p)(0) = (0,0) 
			&&\text{in }\Omega
			\\
			&\overline\nu\partial_t^2 \overline q - \Delta \overline q - \eta\partial_t \overline q = \sum_{i=1}^{m}a_i(\overline p - p_i^{ob})
			&&\text{in }I\times\Omega
			\\
			&\partial_n \overline q = 0 
			&&\text{in }I\times\Gamma_N
			\\
			&\overline q = 0
			&&\text{in }I\times\Gamma_D
			\\
			&(\overline q, \partial_t\overline q)(T) = (0,0)
			&&\text{in }\Omega
			\\
			&\left(-\int_{0}^{T}\partial_t^2\overline p(t)\overline q(t) \d t +\lambda\overline{\nu},\nu-\overline{\nu} \right)_{L^2(\Omega)}\geq 0 
			&&\forall \nu\in\mathcal{V}_{ad}.
		\end{aligned}\right.
	\end{equation}
	Further, we assume that $p_i^{ob}\in W^{4,1}(I,L^2(\Omega))$, $a_i\in C^4(I, L^\infty(\Omega))$, and $\partial_t^l a_i(T) = 0$ for all $i=1,\ldots,m$ and $l=0,1,2,3$.
\end{assumption}

Under the \cref{assumption2}, the application of \cref{lemma: abstract} and \cref{lemma: stampaccia} to \cref{first order optimality} yields the higher regularity properties 
\begin{equation}\label{regularity overline p}
	\overline p\in C^7(I,L^2(\Omega))\cap C^6(I,H^1_D(\Omega))\cap C^5(I,D(\Delta_{D,N}))\cap C^5(I, L^\infty(\Omega))	
\end{equation}
and $\overline q\in C^5(I,L^2(\Omega))\cap C^4(I,H^1_D(\Omega))\cap C^3(I,D(\Delta_{D,N}))\cap C^3(I, L^\infty(\Omega))$. For the derivation and justification of the necessary optimality system \cref{first order optimality}, we refer the reader to \cite[Theorem 3.5]{ammann23}. Further, note that the regularities of $a_i$ are reasonable since, in the application, the product of characteristic functions in space with a smooth function in time is typically considered.

To prepare for the second-order sufficient optimality result, let us define the Lagrangian functional associated with \cref{P} by
\begin{equation}\label{def: lagrangian}
	\mathcal L (\nu,p,q) \coloneqq \mathcal J(\nu,p) - (\nu\partial_t^2 p - \Delta p + \eta \partial_t p -f, q)_{L^2(I,L^2(\Omega))} - (p(0),q(0))_{L^2(\Omega)} - (\partial_tp(0),\partial_tq(0))_{L^2(\Omega)}.
\end{equation}
For $h\in L^2(\Omega)$, we introduce the linearized state equation at $(\overline \nu,\overline p)$ as follows: 
\begin{equation}
	\label{system: statelin}
	\left\{\begin{aligned}
		&\overline\nu\partial_t^2\tilde p - \Delta\tilde p + \eta \partial_t\tilde p = -h\partial_t^2\overline p 
		&&\text{in }I\times \Omega
		\\
		&\partial_n \tilde p = 0 
		&&\text{in }I\times\Gamma_N
		\\
		&\tilde p = 0
		&&\text{in }I\times\Gamma_D
		\\
		&(\tilde p, \partial_t\tilde p)(0) = (0, 0) 
		&& \text{in }\Omega.
	\end{aligned}\right.
\end{equation}
Thanks to \cref{regularity overline p}, the well-posedness of \cref{system: statelin} follows immediately from \cref{lemma: abstract}. Furthermore, let the set of strongly active constraints be given by
\begin{equation}\label{def:A0}
		\mathscr A_{0}(\overline\nu) \coloneqq \left\{x\in\Omega\colon -\int_{0}^{T}\partial_t^2\overline p(t,x) \overline q(t,x)\d t + \lambda\overline\nu(x)\neq 0\right\}
\end{equation}
and let the associated critical cone be given by
\begin{equation*}
	C_{\overline\nu}^0 \coloneqq \{h\in L^2(\Omega) :\text{ For a.e. } x\in\Omega:h(x)\geq 0\text{ if } \overline\nu(x) = \nu_-(x); h(x)\leq 0 \text{ if } \overline\nu(x) = \nu_+(x)\,;
	h\big |_{\mathscr A_0(\overline\nu)} \equiv 0 \}.
\end{equation*}

\begin{theorem}\label{theorem: ssc}
	Let \cref{assumption2} hold. Further, assume that 
	\begin{equation}
		\label{ssc}\tag{SSC}
		\left\{\begin{aligned}
			&D_{(\nu,p)}^2\mathcal L(\overline\nu,\overline p,\overline q)(h,\tilde p)^2>0\quad \forall h\in C_{\overline\nu}^0\setminus\{0\}
			\\
			&\text{where $\tilde p\in C^2(I,L^2(\Omega))\cap C^1(I,H^1_D(\Omega))\cap C(I,D(\Delta_{D,N}))$ denotes the solution to \cref{system: statelin}}.
		\end{aligned}\right.
	\end{equation}
	Then, there exist $\sigma>0$ and $\beta >0$ such that the quadratic growth condition 
	\begin{equation}
		\label{ieq:growthcondition}
		\mathcal J(\nu, p)\geq \mathcal J(\overline\nu, \overline p) + \beta\|\nu-\overline\nu\|_{L^2(\Omega)}^2
	\end{equation}
	holds true for every $\nu\in\mathcal{V}_{ad}$ with $\|\nu-\overline\nu\|_{L^2(\Omega)}\leq\sigma$ and the corresponding solution $p$ to \cref{system: state}. In particular, $\overline\nu$ is a locally optimal solution to \cref{P}.
\end{theorem}

\begin{proof}
	Under \cref{assumption2}, for every $\nu, \hat\nu\in\mathcal V_{ad}$, by applying \cref{lemma: abstract} and \cref{lemma: stampaccia} to the difference of the corresponding solutions $p, \hat p$ to \cref{system: state}, we have that
	\begin{equation}\label{estimate for ssc}
		\|p - \hat p\|_{C^4(I, L^2(\Omega))} + \|p - \hat p\|_{C^2(I, L^\infty(\Omega))}\leq L\|\nu - \hat\nu\|_{L^2(\Omega)}
	\end{equation}
	for a constant $L>0$, independent of $\nu, \hat\nu, p, \hat p$. Thus, the proof follows precisely the one of \cite[Theorem 4.6]{ammann23} by replacing the inner regularity result \cite[Lemma 4.3]{ammann23} with \cref{estimate for ssc}.
\end{proof}

\section{Perturbed optimality system}
\label{section perturbed systems}

This section is devoted to the analysis of perturbed and linearized optimality systems associated with \cref{P}, which will play an essential role in the analysis of the SQP method (see \cref{section SQP}). In the following, let \cref{assumption2} hold. Given some perturbation term $(\rho^{VI},\rho^{st}, \rho^{adj})\in L^\infty(\Omega)\times H^1(I, L^2(\Omega))\times H^1(I, L^2(\Omega))$ with $\rho^{st}(0) = \rho^{adj}(T) = 0$, we consider the system 
\begin{equation}\label{system: linpertOC}\tag{OS}
	\left\{\begin{aligned}
		&\overline\nu\partial_t^2 p - \Delta p + \eta\partial_t p = f - (\nu - \overline\nu)\partial_t^2\overline p + \rho^{st}
		&&\text{in }I\times\Omega
		\\
		&\partial_n p = 0 
		&&\text{in }I\times\Gamma_N
		\\
		&p = 0
		&&\text{in }I\times\Gamma_D
		\\
		&(p, \partial_t p)(0) = (0, 0)
		&&\text{in }\Omega
		\\
		&\overline\nu\partial_t^2 q - \Delta q - \eta\partial_t q = \sum_{i=1}^m a_i(p - p_i^{ob}) - (\nu - \overline\nu)\partial_t^2\overline q + \rho^{adj}
		&&\text{in }I\times\Omega
		\\
		&\partial_n q = 0 
		&&\text{in }I\times\Gamma_N
		\\
		&q = 0
		&&\text{in }I\times\Gamma_D
		\\
		&(q, \partial_t q)(T) = (0, 0)
		&&\text{in }\Omega
		\\
		&\nu\in\mathcal V_{ad}, (-\int_0^T\!\!\!\partial_t^2\overline p(t) q(t) + \partial_t^2(p(t)-\overline p(t))\overline q(t)\d t + \lambda\nu, \tilde\nu - \nu)_{L^2(\Omega)} \geq (\rho^{VI}, \tilde\nu - \nu)_{L^2(\Omega)}
		&&\text{for all }\tilde\nu\in\mathcal V_{ad}.
	\end{aligned}\right.
\end{equation}
Sufficient second-order optimality conditions for \cref{P} are the main ingredients for the analysis of \cref{system: linpertOC}. Unfortunately, the originally proposed \cref{ssc} in \cref{theorem: ssc} is too weak for our purposes, as the involved critical cone $C_{\overline\nu}^0\setminus\{0\}$ is too restrictive. Thus, for a fixed $\tau >0$, we introduce the enlarged critical cone
\begin{equation}\label{def critical cone tau}
	C_{\overline\nu}^{\tau} \coloneqq \{h\in L^2(\Omega) : h(x) = 0 \text{ for a.e. } x\in \mathscr A_\tau(\overline\nu) \},
\end{equation}
where the set of $\tau$-uniform strongly active constraints is given by
\begin{equation}\label{def:Atau}
	\mathscr A_\tau(\overline\nu) \coloneqq \left\{x\in\Omega : \left|-\int_{0}^{T}\partial_t^2\overline p(t,x) \overline q(t,x)\d t + \lambda\overline\nu(x)\right| > \tau \right\}.
\end{equation}
Then, the strengthened sufficient second-order optimality condition reads as follows:
\begin{equation}\label{ssc'}\tag{$\text{SSC}^\tau$}
	\left\{
		\begin{aligned}
			&\text{There exist $\tau>0$ and $\alpha>0$ such that } D_{(\nu,p)}^2\mathcal L(\overline\nu,\overline p,\overline q)(h,\tilde p)^2>\alpha\|h\|_{L^2(\Omega)}^2 \text{for every $h\in C^\tau_{\overline\nu}$}
			\\
			&\text{where $\tilde p\in C^2(I,L^2(\Omega))\cap C^1(I,H^1_D(\Omega))\cap C(I,D(\Delta_{D,N}))$ denotes the solution to \cref{system: statelin}.}
		\end{aligned}
	\right.	
\end{equation}
Note that $\mathscr A_\tau(\overline\nu) \subset \mathscr A_0(\overline\nu)$ and therefore the new critical cone $C_{\overline\nu}^{\tau}$ is in fact enlarged, i.e. $C_{\overline\nu}^0\subset C_{\overline\nu}^{\tau}$. Therefore, the strengthened sufficient second-order condition \cref{ssc'} particularly implies the original condition \cref{ssc}. As a consequence, it also guarantees the optimality of $\overline\nu$ along with the quadratic growth condition \cref{ieq:growthcondition} under the same assumptions as in \cref{theorem: ssc}. Still, with the strengthened condition \cref{ssc'} the following difficulty appears: Given two controls $\nu_1, \nu_2\in \mathcal V_{ad}$, the difference $h = \nu-\overline\nu$ for $\nu\in\mathcal V_{ad}$ does not belong to the enlarged critical cone $C_{\overline\nu}^\tau$. Therefore, it does not satisfy the assumptions of \cref{ssc'}. To circumvent this difficulty, we extend well-known techniques (see \cite{troeltzsch99,volkwein03,wachsmuth07,hoppe21}) to our hyperbolic case. We consider an auxiliary problem by replacing the admissible set $\mathcal V_{ad}$ with 
\begin{equation}\label{def: Vad0}
	\mathcal V_{ad}^\tau \coloneqq \{\nu\in\mathcal V_{ad} \,\,|\,\,\nu = \overline\nu \text{ a.e. in }\mathscr A_\tau(\overline\nu)\}.
\end{equation}
Now, given two controls $\nu,\widetilde\nu\in\mathcal V_{ad}^\tau$, the difference $h = \nu-\widetilde\nu$ fullfils $h\in C_{\overline\nu}^\tau$. We define the following modification of the perturbed linearized optimality system \cref{system: linpertOC}:
\begin{equation}\label{system: linpertOC2}\tag{$\text{OS}^\tau$}
	\cref{system: linpertOC} \text{ where $\mathcal V_{ad}$ is replaced with $\mathcal V_{ad}^\tau$.}
\end{equation}

\begin{proposition}\label{proposition: well-definedness}
	Let \cref{assumption2} and \cref{ssc'} be satisfied. Then, for all $(\rho^{st}, \rho^{adj}, \rho^{VI})\in H^1(I, L^2(\Omega))\times H^1(I, L^2(\Omega))\times L^2(\Omega)$ with $\rho^{st}(0) = \rho^{adj}(T) = 0$, the system \cref{system: linpertOC2} admits a unique solution $(\nu,p,q)\in\mathcal V_{ad}^\tau\times C^2(I, L^2(\Omega))\cap C^1(I, H^1_D(\Omega))\cap C(I, D(\Delta_{D,N})) \times C^2(I, L^2(\Omega))\cap C^1(I, H^1_D(\Omega))\cap C(I, D(\Delta_{D,N}))$.
\end{proposition}

\begin{proof}
	Let $(\rho^{st}, \rho^{adj}, \rho^{VI})\in H^1(I, L^2(\Omega))\times H^1(I, L^2(\Omega))\times L^2(\Omega)$ with $\rho^{st}(0) = \rho^{adj}(T) = 0$ be given. Thanks to \cref{regularity overline p}, \cref{lemma: abstract} implies that 
	\begin{equation}\label{system: linearized pertubed}
		\left\{\begin{aligned}
			&\overline\nu\partial_t^2 p - \Delta p + \eta\partial_t p = f - (\nu - \overline\nu)\partial_t^2\overline p + \rho^{st}
			&&\text{in }I\times\Omega
			\\
			&\partial_n p = 0 
			&&\text{in }I\times\Gamma_N
			\\
			&p = 0
			&&\text{in }I\times\Gamma_D
			\\
			&(p, \partial_t p)(0) = (0, 0)
			&&\text{in }\Omega
		\end{aligned}\right.
	\end{equation}
	admits for every $\nu\in L^2(\Omega)$ a unique solution $p\in C^2(I, L^2(\Omega))\cap C^1(I, H^1_D(\Omega))\cap C(I, D(\Delta_{D, N}))$. Denoting by $S_\rho\colon L^2(\Omega)\to C^2(I, L^2(\Omega))\cap C^1(I, H^1_D(\Omega))\cap C(I, D(\Delta_{D,N})), \nu\mapsto p$ the affine linear and continuous solution operator to \cref{system: linearized pertubed}, we consider the minimization problem
	\begin{align} \label{minimization}
		\min_{\nu\in\mathcal V_{ad}^\tau} J_\rho(\nu)&\coloneqq \mathcal J(\nu,S_\rho(\nu)) + \left(-\int_0^T\partial_t^2 (S_\rho(\nu) - \overline p)\overline q\d t - \rho^{VI},\nu - \overline\nu\right)_{L^2(\Omega)} + (\rho^{adj},S_\rho(\nu))_{L^2(I,L^2(\Omega))} 
		\\\notag
		&\underbrace{=}_{\cref{P}} \frac{1}{2} \sum_{i=1}^m \int_{0}^{T}\int_{\Omega}a_i(S_\rho(\nu )-p^{ob}_i)^2\d x\d t+\frac{\lambda}{2}\|\nu\|_{L^2(\Omega)}^2 + \left(-\int_0^T\partial_t^2 (S_\rho(\nu) - \overline p)\overline q\d t - \rho^{VI},\nu - \overline\nu\right)_{L^2(\Omega)}
		\\\notag
		&\quad + (\rho^{adj},S_\rho(\nu))_{L^2(I,L^2(\Omega))}.
	\end{align}
	By the quadratic structure of $J_\rho$, we have that
	\begin{equation}\label{taylorexpansion Jrho}
		J_\rho(\tilde\nu) = J_\rho(\nu) +J'_\rho(\nu)(\tilde\nu - \nu) + \frac{1}{2}J''_\rho(\nu)(\tilde\nu - \nu)^2\quad\forall \nu,\tilde\nu\in\mathcal V_{ad}^\tau.
	\end{equation}
	Further, for any $\nu,\tilde\nu\in\mathcal V_{ad}^\tau$, it holds that $h\coloneqq \tilde\nu - \nu\in C^\tau_{\overline\nu}$ (see \cref{def critical cone tau}) and $\tilde p\coloneqq S_\rho(\tilde\nu) - S_\rho(\nu)$ solves the linearized state equation \cref{system: statelin} such that \cref{ssc'} yields that 
	\begin{align*}
		\alpha\|h\|_{L^2(\Omega)}^2<D_{(\nu,p)}^2\mathcal L(\overline\nu,\overline p,\overline q)(h,\tilde p)^2 \underbrace{=}_{\cref{def: lagrangian}} \sum_{i=1}^m(a_i\tilde p,\tilde p)_{L^2(I,L^2(\Omega))} + \lambda\|h\|_{L^2(\Omega)}^2 - 2 (h\partial_t^2\tilde p,\overline q)_{L^2(I,L^2(\Omega))} 
		\underbrace{=}_{\cref{minimization}} J''_\rho(\nu)h^2,
	\end{align*}
	where for the last equality we have used the fact that $S_\rho$ is continuous and affine linear such that
	\begin{equation*}
		S_\rho'(\nu) h = S_\rho'(\nu) (\tilde \nu - \nu) = S_\rho(\tilde\nu) - S_\rho(\nu) = \tilde p \quad \text{and} \quad S_\rho''(\nu) h^2=0.	
	\end{equation*}
	Therefore, along with \cref{taylorexpansion Jrho}, the strict convexity of $J_\rho$ in $\mathcal V_{ad}^\tau$ is obtained. In conclusion, together with the continuity of $J_\rho\colon L^2(\Omega)\to\mathbb R$, \cref{minimization} admits a unique solution $\nu\in \mathcal V_{ad}^\tau$. Moreover, the necessary and sufficient optimality condition for \cref{minimization} is given by $\nu \in \mathcal V_{ad}^\tau$ and $J'_\rho(\nu)(\tilde\nu - \nu)\geq 0$ for every $\tilde\nu\in\mathcal V_{ad}^\tau$, which is equivalent to \cref{system: linpertOC2} due to standard arguments. Thus, the claim is valid. 
\end{proof}

The following \cref{theorem:lipschitzproperty} provides the crucial stability result for the solution to \cref{system: linpertOC2} regarding the perturbation terms. In the proof, we extend known ideas incorporating the first- and second-order optimality conditions (cf. \cite{troeltzsch01,volkwein03}) to our given case. Here, the effect of loss of regularity causes some challenges that we handle by using the stronger norm $\|\cdot\|_{H^1(I, L^2(\Omega))}$ instead of $\|\cdot\|_{L^2(I, L^2(\Omega))}$ at the right hand side of the estimates in \cref{Lipschitzproperties2,Lipschitzproperties3}.

\begin{theorem}\label{theorem:lipschitzproperty}
	Let \cref{assumption2} and \cref{ssc'} hold. 
	Then, there exist constants $L, L_p, L_q>0$ such that for all perturbation terms $(\rho^{st}, \rho^{adj},\rho^{VI}), (\widetilde\rho^{st}, \widetilde\rho^{adj},\widetilde\rho^{VI})\in H^1(I, L^2(\Omega))\times H^1(I, L^2(\Omega))\times L^2(\Omega)$ with $\rho^{st}(0) = \widetilde\rho^{st}(0) = \rho^{adj}(T) = \widetilde\rho^{adj}(T) = 0$, the corresponding solutions $(\nu_\rho,p_\rho,q_\rho)$ and $(\nu_{\widetilde\rho},p_{\widetilde\rho},q_{\widetilde\rho})$ to \cref{system: linpertOC2} satisfy
	\begin{align}\label{Lipschitzproperties}
		\|\nu_\rho - \nu_{\widetilde\rho}\|_{L^2(\Omega)}
		&\leq L(\|\rho^{st} - \widetilde\rho^{st}\|_{L^2(I, L^2(\Omega))} + \|\rho^{adj} - \widetilde\rho^{adj}\|_{L^2(I, L^2(\Omega))} + \|\rho^{VI} - \widetilde\rho^{VI}\|_{L^2(\Omega)})
		\\\label{Lipschitzproperties2}
		\|p_\rho - \overline p\|_{L^2(I,L^\infty(\Omega))} 
		&\leq L_p(\|\rho^{st}\|_{H^1(I, L^2(\Omega))} + \|\rho^{adj}\|_{L^2(I, L^2(\Omega))} + \|\rho^{VI}\|_{L^2(\Omega)})
		\\\label{Lipschitzproperties3}
		\|q_\rho - \overline q\|_{L^2(I,L^\infty(\Omega))} 
		&\leq L_q(\|\rho^{st}\|_{L^2(I, L^2(\Omega))} + \|\rho^{adj}\|_{H^1(I, L^2(\Omega))} + \|\rho^{VI}\|_{L^2(\Omega)}).
	\end{align}
\end{theorem}

\begin{proof}
	Let $(\nu_\rho,p_\rho,q_\rho), (\nu_{\widetilde\rho},p_{\widetilde\rho},q_{\widetilde\rho})\in \mathcal V_{ad}^\tau\times (C^2(I, L^2(\Omega))\cap C^1(I, H^1_D(\Omega))\cap C(I, D(\Delta_{D,N})))^2$ denote the unique solutions to \cref{system: linpertOC2} with respect to $(\rho^{st},\rho^{adj}, \rho^{VI})$ and $(\widetilde\rho^{VI},\widetilde\rho^{st},\widetilde\rho^{adj})$ according to \cref{proposition: well-definedness}. Subtracting the corresponding PDE-systems (see \cref{system: linpertOC}), we obtain that 
	\begin{align}
		\label{system: statedifference}
		&\left\{\begin{aligned}
			&\overline\nu\partial_t^2(p_\rho- p_{\widetilde\rho}) - \Delta(p_\rho- p_{\widetilde\rho}) + \eta\partial_t(p_\rho- p_{\widetilde\rho}) = -(\nu_\rho- \nu_{\widetilde\rho})\partial_t^2\overline p + \rho^{st}-\widetilde\rho^{st}
			&&\text{in }I\times \Omega
			\\
			&\partial_n(p_\rho- p_{\widetilde\rho}) = 0
			&&\text{in }I\times\Gamma_N
			\\
			&p_\rho- p_{\widetilde\rho} = 0
			&&\text{in }I\times\Gamma_D
			\\
			&(p_\rho- p_{\widetilde\rho},\partial_t(p_\rho- p_{\widetilde\rho}))(0) = (0, 0) && \text{in }\Omega
		\end{aligned}\right.
		\\
		\label{system: adjointdifference}
		&\left\{\begin{aligned}
			&\overline\nu\partial_t^2(q_\rho- q_{\widetilde\rho}) - \Delta(q_\rho- q_{\widetilde\rho}) - \eta\partial_t(q_\rho- q_{\widetilde\rho}) = \sum_{i=1}^m a_i(p_\rho - p_{\widetilde\rho}) -(\nu_\rho- \nu_{\widetilde\rho})\partial_t^2\overline q + \rho^{adj}-\widetilde\rho^{adj}
			&&\text{in }I\times \Omega
			\\
			&\partial_n(q_\rho- q_{\widetilde\rho}) = 0
			&&\text{in }I\times\Gamma_N
			\\
			&q_\rho- q_{\widetilde\rho} = 0
			&&\text{in }I\times\Gamma_D
			\\
			&(q_\rho- q_{\widetilde\rho},\partial_t(q_\rho- q_{\widetilde\rho}))(T) = (0, 0) 
			&& \text{in }\Omega.
		\end{aligned}\right.
	\end{align}
	We begin by elaborating on the control parameter. By the construction of $\mathcal V_{ad}^\tau$ (see \cref{def critical cone tau}), the quantity $h\coloneqq\nu_\rho - \nu_{\widetilde\rho}$ lies in the critical cone $C_{\nu}^\tau$, and $p_\rho - p_{\widetilde\rho} - \hat p_{\rho,\widetilde\rho}$ solves the associated linearized state equation \cref{system: statelin} where $\hat p_{\rho,\widetilde\rho}$ denotes the solution to 
	\begin{equation}\label{prhorho}
		\left\{\begin{aligned}
			&\overline\nu\partial_t^2\hat p_{\rho,\widetilde\rho} - \Delta\hat p_{\rho,\widetilde\rho} + \eta\partial_t\hat p_{\rho,\widetilde\rho} = \rho^{st}-\widetilde\rho^{st}
			&&\text{in }I\times \Omega
			\\
			&\partial_n \hat p_{\rho,\widetilde\rho} = 0 
			&&\text{in }I\times\Gamma_N
			\\
			&\hat p_{\rho,\widetilde\rho} = 0
			&&\text{in }I\times\Gamma_D
			\\
			&(\hat p_{\rho,\widetilde\rho},\partial_t\hat p_{\rho,\widetilde\rho})(0) = (0, 0) && \text{in }\Omega.
		\end{aligned}\right.
	\end{equation}
	Thus, \cref{ssc'} yields that	
	\begin{align}
		\label{ieq:nurho00}
			\alpha\|\nu_\rho - \nu_{\widetilde\rho}\|_{L^2(\Omega)}^2\leq D_{(\nu,p)}^2\mathcal L(\overline \nu, \overline p, \overline q)(\nu_\rho - \nu_{\widetilde\rho},p_\rho - p_{\widetilde\rho} - \hat p_{\rho,\widetilde\rho})^2
	\end{align} 
	Utilizing \cref{ieq:nurho00} in combination with \cref{lemma: abstract}, we come to the conclusion that \cref{Lipschitzproperties} holds. The detailed proof for this can be found in the \cref{subsection proof pertubation}. To prove \cref{Lipschitzproperties2}, notice that $p_\rho - \overline p$ is the solution to
	\begin{equation*}
		\left\{\begin{aligned}
			&\overline\nu\partial_t^2(p_\rho - \overline p) - \Delta(p_\rho - \overline p) + \eta\partial_t(p_\rho - \overline p) = - (\nu_\rho - \overline\nu)\partial_t^2\overline p + \rho^{st} 
			&&\text{in }I\times\Omega
			\\
			&\partial_n (p_\rho - \overline p) = 0
			&&\text{in }I\times\Gamma_N
			\\
			&p_\rho - \overline p = 0
			&&\text{in }I\times\Gamma_D
			\\
			&(p_\rho - \overline p,\partial_t(p_\rho - \overline p))(0) = (0,0)
			&&\text{in }\Omega.
		\end{aligned}\right.
	\end{equation*}
	Applying \cref{lemma: stampaccia} to the above system yields that
	\begin{align}\label{ieq p rho overline p}
		&\|p_\rho - \overline p\|_{L^2(I,L^\infty(\Omega))}
		\\\notag
		&\leq \hat c (\|(\nu_\rho - \overline\nu)\partial_t^2\overline p\|_{L^2(I,L^2(\Omega))} + \|\rho^{st}\|_{L^2(I,L^2(\Omega))}+ \|(\nu_\rho - \overline\nu)\partial_t^3\overline p\|_{L^2(I,L^2(\Omega))}+ \|\partial_t\rho^{st}\|_{L^2(I,L^2(\Omega))})
		\\\notag
		&\leq \hat c((\|\partial_t^2\overline p\|_{L^2(I,L^\infty(\Omega))} + \|\partial_t^3\overline p\|_{L^2(I,L^\infty(\Omega))})\|\nu_\rho - \overline\nu\|_{L^2(\Omega)}+ \|\rho^{st}\|_{L^2(I,L^2(\Omega))} + \|\partial_t\rho^{st}\|_{L^2(I,L^2(\Omega))}).
	\end{align}
	Since $(\overline\nu, \overline p, \overline q)$ solves \cref{system: linpertOC2} with $(\rho^{st}, \rho^{adj},\rho^{vi}) = 0$, applying \cref{Lipschitzproperties} to \cref{ieq p rho overline p} leads to \cref{Lipschitzproperties2}. Analogously, applying \cref{lemma: abstract} to the above system, we obtain with $G(t)\coloneqq \int_0^{t} g(s) \d s$ and $g(s) \coloneqq - (\nu_\rho - \overline\nu)\partial_t^2\overline p (s) + \rho^{st} (s)$ that 
	\begin{align}\label{ieq: alpha3proof}
		\| p_\rho - \overline p \|_{L^2(I,L^2(\Omega))} 
		&\leq \sqrt{T}c\|G\|_{L^1(I, L^2(\Omega))}\leq T^{\frac{3}{2}}c( \| \nu_\rho - \overline\nu \|_{L^2(\Omega)} \| \partial_t^2\overline p\|_{L^1(I, L^\infty(\Omega))} + \|\rho^{st} \|_{L^1(I, L^2(\Omega))} ) 
		\\\notag
		\| \partial_t(p_\rho - \overline p) \|_{L^2(I,L^2(\Omega))} 
		&\leq \sqrt{T}c\|g\|_{L^1(I, L^2(\Omega))}\leq \sqrt{T}c( \| \nu_\rho - \overline\nu \|_{L^2(\Omega)} \|\partial_t^2\overline p \|_{L^1(I, L^\infty(\Omega))} + \| \rho^{st} \|_{L^1(I, L^2(\Omega))}).
	\end{align}
	Since $q_\rho - \overline q$ solves
	\begin{equation*}
		\left\{\begin{aligned}
			&\overline\nu\partial_t^2(q_\rho - \overline q) - \Delta(q_\rho - \overline q) - \eta\partial_t(q_\rho - \overline q) = \sum_{i=1}^{m}a_i(p_\rho - \overline p)- (\nu_\rho - \overline\nu)\partial_t^2\overline q + \rho^{adj} 
			&&\text{in }I\times\Omega
			\\
			&\partial_n (q_\rho - \overline q) = 0 
			&&\text{in }I\times\Gamma_N
			\\
			&q_\rho - \overline q = 0
			&&\text{in }I\times\Gamma_D
			\\
			&(q_\rho - \overline q,\partial_t(q_\rho - \overline q))(T) = (0,0)
			&&\text{in }\Omega,
		\end{aligned}\right.
	\end{equation*}
	\cref{lemma: stampaccia} yields that
	\begin{align*}
		&\|q_\rho - \overline q\|_{L^2(I,L^\infty(\Omega))}
		\\
		&\leq \hat c( \sum_{i=1}^{m}\|a_i\|_{L^\infty(I, L^\infty(\Omega))}(\|p_\rho - \overline p\|_{L^2(I,L^2(\Omega))} + \|\partial_t(p_\rho - \overline p)\|_{L^2(I,L^2(\Omega))}) \!+\! \sum_{i=1}^{m}\|\partial_ta_i\|_{L^\infty(I, L^\infty(\Omega))}\|p_\rho - \overline p\|_{L^2(I,L^2(\Omega))} 
		\\
		&\hspace{.8cm}+ (\|\partial_t^2\overline q\|_{L^2(I,L^\infty(\Omega))} + \|\partial_t^3\overline q\|_{L^2(I,L^\infty(\Omega))} )\|\nu_\rho - \overline\nu\|_{L^2(\Omega)}+ \|\rho^{adj}\|_{L^2(I,L^2(\Omega))} + \|\partial_t\rho^{adj}\|_{L^2(I,L^2(\Omega))}).
	\end{align*}
	Applying \cref{Lipschitzproperties} with $(\widetilde\rho^{st}, \widetilde\rho^{adj}, \widetilde\rho^{vi}) = 0$ and \cref{ieq: alpha3proof}, we obtain \cref{Lipschitzproperties3}.
\end{proof}

With the following lemma, we will abandon the modification $\mathcal V_{ad}^\tau$ of the admissible set $\mathcal V_{ad}$. The proof follows the argumentation from \cite[Corollary 5.3]{wachsmuth07} with a careful modification.

\begin{lemma}\label{lemma: firstorderrho}
	Let \cref{assumption2} and \cref{ssc'} hold. Let $(\rho^{st}, \rho^{adj}, \rho^{VI})\in H^1(I, L^2(\Omega))\times H^1(I, L^2(\Omega))\times L^\infty(\Omega)$ such that $\rho^{st}(0) = \rho^{adj}(T) = 0$ and 
	\begin{equation}\label{ieq: perturbation bound}
		\|\rho^{st}\|_{H^1(I, L^2(\Omega))} + \|\rho^{adj}\|_{H^1(I, L^2(\Omega))} + \|\rho^{VI}\|_{L^\infty(\Omega)}\leq \frac{\tau}{c_L},\quad c_L\coloneqq \max\{L_p\|\partial_t^2\overline q\|_{L^2(I,L^\infty(\Omega))}, L_q\|\partial_t^2\overline p\|_{L^2(I,L^\infty(\Omega))}, 1\}.
	\end{equation}
	Then, the unique solution to \cref{system: linpertOC2} satisfies \cref{system: linpertOC}. 
\end{lemma}

\begin{proof}
	Let $(\nu_\rho,p_\rho,q_\rho)$ denote the solution to \cref{system: linpertOC2}. Since the equations in \cref{system: linpertOC} and \cref{system: linpertOC2} coincide, it remains to show that the variational inequality in \cref{system: linpertOC} is valid. By \cref{system: linpertOC2}, it holds that
	\begin{equation}\label{vi:psi01}
		(-\int_0^T\partial_t^2\overline p(t) q_\rho(t) + \partial_t^2(p_\rho(t)-\overline p(t))\overline q(t)\d t + \lambda\nu_\rho, \widetilde\nu - \nu_\rho)_{L^2(\Omega)} \geq (\rho^{VI}, \widetilde\nu - \nu_\rho)_{L^2(\Omega)}\quad \forall\widetilde\nu\in\mathcal V_{ad}^\tau.
	\end{equation}
	By the definition of $\mathcal V_{ad}^\tau$ (see \cref{def: Vad0}) and since $\nu_\rho\in\mathcal V_{ad}^\tau$, it holds for every $\widetilde\nu\in\mathcal V_{ad}^\tau$ that $\widetilde\nu - \nu_\rho = 0$ a.e. in $\mathscr A_{\tau}(\overline\nu)$. As a consequence, \cref{vi:psi01} implies that
	\begin{equation}\label{vi:psi02}
		(-\int_0^T\partial_t^2\overline p(t)q_\rho(t) + \partial_t^2(p_\rho(t)-\overline p(t))\overline q(t)\d t + \lambda\nu_\rho, \widetilde\nu - \nu_\rho)_{L^2(\Omega\setminus\mathscr A_\tau(\overline\nu))} \geq (\rho^{VI}, \widetilde\nu - \nu_\rho)_{L^2(\Omega\setminus\mathscr A_\tau(\overline\nu))}\,\, \forall\widetilde\nu\in\mathcal V_{ad}^\tau.
	\end{equation}
	For every $\nu\in\mathcal V_{ad}$, we set $\widetilde\nu\coloneqq \chi_{\mathscr A_\tau(\overline\nu)}\overline\nu + \chi_{(\Omega\setminus\mathscr A_\tau(\overline\nu))}\nu\in \mathcal V_{ad}^\tau$ in \cref{vi:psi02}. Since $\widetilde\nu$ and $\nu$ coincide in $\Omega\setminus\mathscr A_\tau(\overline\nu)$, it follows that \cref{vi:psi02} holds for every $\nu\in\mathcal V_{ad}$, i.e.,
	\begin{equation}\label{vi:psi02b}
		(-\int_0^T\partial_t^2\overline p(t)q_\rho(t) + \partial_t^2(p_\rho(t)-\overline p(t))\overline q(t)\d t + \lambda\nu_\rho, \nu - \nu_\rho)_{L^2(\Omega\setminus\mathscr A_\tau(\overline\nu))} \geq (\rho^{VI}, \nu - \nu_\rho)_{L^2(\Omega\setminus\mathscr A_\tau(\overline\nu))}\,\, \forall\nu\in\mathcal V_{ad}.
	\end{equation}
	Let $\mathscr A_\tau^+(\overline\nu)\coloneqq \{x\in\Omega : -\int_{0}^{T}\partial_t^2\overline p(t,x)\overline q(t,x)\d t + \lambda\overline\nu(x) >\tau \}$ and $\mathscr A_\tau^-(\overline\nu)\coloneqq \{x\in\Omega : -\int_{0}^{T}\partial_t^2\overline p(t,x)\overline q(t,x)\d t + \lambda\overline\nu(x) <-\tau \}$. Then, by \cref{def:Atau}, it holds that $\mathscr A_\tau(\overline\nu) = \mathscr A_\tau^+(\overline\nu)\cup\mathscr A_\tau^-(\overline\nu)$, and it follows for a.e. $x\in\mathscr A_\tau^+(\overline\nu)$ that
	\begin{align}\label{ieq tau}
		\tau &<-\int_{0}^{T}\partial_t^2\overline p(t,x) \overline q(t,x)\d t + \lambda\overline\nu(x)
		\\
		&\notag= -\int_0^T\partial_t^2\overline p(t,x)q_\rho(t,x) + \partial_t^2(p_\rho(t,x)+\overline p(t,x))\overline q(t,x)\d t + \lambda\nu_\rho(x) - \rho^{VI}(x)
		\\
		&\notag\quad + \int_0^T \partial_t^2\overline p(t,x)(q_\rho(t,x) - \overline q(t,x)) + \partial_t^2(p_\rho(t,x)- \overline p(t,x))\overline q(t,x)\d t + \rho^{VI}(x)
		\\
		&\notag\leq -\int_0^T\partial_t^2\overline p(t,x)q_\rho(t,x) + \partial_t^2(p_\rho(t,x)+\overline p(t,x))\overline q(t,x)\d t + \lambda\nu_\rho(x) - \rho^{VI}(x)
		\\
		&\notag\quad + \|\partial_t^2\overline p\|_{L^2(I,L^\infty(\Omega))}\|q_\rho -\overline q\|_{L^2(I,L^\infty(\Omega))} + \|p_\rho - \overline p\|_{L^2(I,L^\infty(\Omega))}\|\partial_t^2\overline q\|_{L^2(I,L^\infty(\Omega))} + \|\rho^{VI}\|_{L^\infty(\Omega)}
		\\
		&\notag\hspace{-.4cm}\underbrace{\leq}_{\text{Thm.}\,\ref{theorem:lipschitzproperty}} -\int_0^T\partial_t^2\overline p(t,x)q_\rho(t,x) + \partial_t^2(p_\rho(t,x)+\overline p(t,x))\overline q(t,x)\d t + \lambda\nu_\rho(x) - \rho^{VI}(x)
		\\
		&\notag\quad +c_L(\|\rho^{st}\|_{H^1(I, L^2(\Omega))} + \|\rho^{adj}\|_{H^1(I, L^2(\Omega))} + \|\rho^{VI}\|_{L^\infty(\Omega)})
	\end{align}
	where $c_L =\max\{L_p\|\partial_t^2\overline q\|_{L^2(I,L^\infty(\Omega))}, L_q\|\partial_t^2\overline p\|_{L^2(I,L^\infty(\Omega))}, 1\}$. Consequently, we obtain that
	\begin{align}\label{ieq larger}
		0&\underbrace{\leq}_{\cref{ieq: perturbation bound}} \tau - c_L(\|\rho^{st}\|_{H^1(I, L^2(\Omega))} + \|\rho^{adj}\|_{H^1(I, L^2(\Omega))} + \|\rho^{VI}\|_{L^\infty(\Omega)}) 
		\\
		&\notag\underbrace{\leq}_{\cref{ieq tau}} -\int_0^T\partial_t^2\overline p(t,x)q_\rho(t,x) + \partial_t^2(p_\rho(t,x)+\overline p(t,x))\overline q(t,x)\d t + \lambda\nu_\rho(x) - \rho^{VI}(x)\,\,\,\,\,\text{for a.e. }x\in\mathscr A_\tau^+(\overline\nu).
	\end{align}
	Analogously, 
	\begin{equation}\label{ieq smaller}
		0\geq -\int_0^T\partial_t^2\overline p(t,x)q_\rho(t,x) + \partial_t^2(p_\rho(t,x)+\overline p(t,x))\overline q(t,x)\d t + \lambda\nu_\rho(x) - \rho^{VI}(x)\quad\text{for a.e. }x\in\mathscr A_\tau^-(\overline\nu).
	\end{equation}
	On the other hand, by a standard argumentation, due to \cref{first order optimality}, the pointwise inequality
	\begin{align*}
		\left(-\int_0^T\partial_t^2\overline p(t,x) \overline q(t,x)\d t + \lambda\overline\nu(x)\right)(v - \overline\nu(x))\geq 0\quad\text{for all } v\in[\nu_-(x), \nu_+(x)]\text{ and a.e. }x\in\Omega
	\end{align*}
	holds, implying that $\overline\nu = \nu_-$ in $\mathscr A_\tau^+(\overline\nu)$ and $\overline\nu = \nu_+$ in $\mathscr A_\tau^-(\overline\nu)$. Therefore, since $\nu_{\rho} = \overline\nu$ in $\mathscr A_\tau(\overline\nu)$, along with \cref{ieq larger} and \cref{ieq smaller}, we obtain that
	\begin{align}\label{vi:psi03}
			&\left(-\int_0^T\partial_t^2\overline p(t)q_\rho(t) + \partial_t^2(p_\rho(t)-\overline p(t))\overline q(t)\d t + \lambda\nu_\rho - \rho^{VI},\nu-\nu_{\rho}\right)_{L^2(\mathscr A_\tau(\overline\nu))}
			\\\notag
			&=\bigg(\underbrace{-\int_0^T\partial_t^2\overline p(t)q_\rho(t) + \partial_t^2(p_\rho(t)-\overline p(t))\overline q(t)\d t + \lambda\nu_\rho - \rho^{VI}}_{\geq 0\text{ a.e. in $\mathscr A_\tau^+(\overline\nu)$ due to \cref{ieq larger}}},\underbrace{\nu-\nu_-}_{\geq 0\text{ a.e.}}\bigg)_{L^2(\mathscr A_\tau^+(\overline\nu))} 
			\\\notag
			&\quad + \bigg(\underbrace{-\int_0^T\partial_t^2\overline p(t)q_\rho(t) + \partial_t^2(p_\rho(t)-\overline p(t))\overline q(t)\d t + \lambda\nu_\rho - \rho^{VI}}_{\leq 0\text{ a.e. in $\mathscr A_\tau^-(\overline\nu)$ due to \cref{ieq smaller}}},\underbrace{\nu-\nu_+}_{\leq 0\text{ a.e.}}\bigg)_{L^2(\mathscr A_\tau^-(\overline\nu))}\geq 0\quad\forall \nu\in\mathcal V_{ad}.
	\end{align}
	Combining \cref{vi:psi02b} and \cref{vi:psi03} proves the assertion.
\end{proof}

\section{Sequential Quadratic Programming}
\label{section SQP}

The SQP method (cf. \cite[Section 4.11]{troeltzsch10}) approximates \cref{P} by a sequence of coupled systems arising from a suitable linearization process of the optimality system \cref{first order optimality}. This leads to the following algorithm:

\begin{algorithm}[H]
	\caption{Sequential Quadratic Programming}
	\label{alg:sqp}
	\begin{algorithmic}[1]
		\State\label{line1} Choose $(\nu_0,p_0,q_0)$ and set $k = 0$. 
		\State\label{line2} Find $\nu\in\mathcal V_{ad}$ and $p, q\in C^2(I, L^2(\Omega))\cap C^1(I, H^1_D(\Omega))\cap C(I, D(\Delta_{D,N}))$ such that
		\begin{equation}\label{iteration}\tag{$\mathbb P_k$}
			\left\{\begin{aligned}
				&\nu_k\partial_t^2p - \Delta p + \eta\partial_t p = f - (\nu - \nu_k)\partial_t^2 p_k
				&&\text{in }I\times\Omega
				\\
				&\partial_n p = 0
				&&\text{in }I\times\Gamma_N
				\\
				&p = 0
				&&\text{in }I\times\Gamma_D
				\\
				&(p, \partial_t p)(0) = (0,0)
				&&\text{in }\Omega
				\\
				&\nu_k\partial_t^2q - \Delta q - \eta\partial_t q = \sum_{i=1}^m a_i(p - p_i^{ob}) - (\nu - \nu_k)\partial_t^2 q_k
				&&\text{in }I\times\Omega
				\\
				&\partial_n q = 0
				&&\text{in }I\times\Gamma_N
				\\
				&q = 0
				&&\text{in }I\times\Gamma_D
				\\
				&(q, \partial_t q)(T) = (0,0)
				&&\text{in }\Omega
				\\
				&(-\int_0^T\partial_t^2 p_k(t) q(t) + \partial_t^2(p(t)- p_k(t))q_k(t)\d t + \lambda\nu, \tilde\nu - \nu)_{L^2(\Omega)} \geq 0
				&&\text{for all }\tilde\nu\in\mathcal V_{ad},
			\end{aligned}\right.
		\end{equation}
		and set $(\nu_{k+1}, p_{k+1}, q_{k+1})\coloneqq (\nu, p, q)$.
		\State\label{line3} Stop or set $k = k+1$ and go back to step \ref{line2}. 
	\end{algorithmic}
\end{algorithm}

\begin{remark}\label{remark: GE}
	The hyperbolicity and the second-order bilinear character of the PDEs in \cref{iteration} leads to an undesired effect of loss of regularity, causing two major challenges: 
	\begin{enumerate}[(i)]
		\item For a given iterate $(\nu_k,p_k,q_k)\in\mathcal V_{ad}\times C^l(I, L^2(\Omega))\times C^l(I, L^2(\Omega))$ for some $l>2$, the solutions $p_{k+1},q_{k+1}$ to \cref{iteration} are in general only $l-1$-times continuously differentiable. This can be inferred from \cref{lemma: abstract} due to the regularity of the source terms $(\nu - \nu_k)\partial_t^2 p_k,(\nu - \nu_k)\partial_t^2 q_k\in C^{l-2}(I, L^2(\Omega))$ in the PDEs of \cref{iteration}. For this reason, \cref{alg:sqp} is generally only executable for a limited number of iterations. To tackle this issue, we propose using a smooth-in-time regularity condition (see \cref{assumption3}). 
		\item In the parabolic case (cf. \cite{wachsmuth07,troeltzsch99,hintermueller06,hoppe21}), the convergence analysis strongly relies on Robinson's notion of \textit{strong regularity} (see \cite{robinson80}). Unfortunately, the regularity results and estimation for the hyperbolic case (see \cref{lemma: abstract} or \cite[p. 410]{Evans2010}) are weaker than those for the parabolic one. Consequently, the developed strategies for parabolic scenarios cannot be directly transferred to our case and require a substantial extension. 
	\end{enumerate}
\end{remark}

To circumvent (i), we propose the use of smooth-in-time functions with vanishing initial and final conditions:
\begin{align*}\label{defX}
	&X_0\coloneqq \{\psi \in C^\infty(I, L^\infty(\Omega)): \partial_t^l \psi(0) = 0\,\,\text{for all }l\in\mathbb N_0\}
	\\\notag
	&X_T\coloneqq \{\psi \in C^\infty(I, L^\infty(\Omega)): \partial_t^l \psi(T) = 0\,\,\text{for all }l\in\mathbb N_0\}
\end{align*}
with $\mathbb N_0\coloneqq\mathbb N\cup\{0\}$.

\begin{assumption}\label{assumption3}
	Let \cref{assumption2} hold. Furthermore, let $f\in C^\infty(I, L^2(\Omega))$ with $\partial_t^lf(0)= 0$ for all $l\in\mathbb N_0$, let $a_i\in X_T$ for all $i=1,\dots m$, let $p_i^{ob}\in C^\infty(I,L^2(\Omega))$ for all $i=1,\ldots m$, and let $(\nu_0,p_0,q_0)\in\mathcal V_{ad}\times X_0\times X_T$.
\end{assumption}

In view of \cref{assumption3}, \cref{lemma: stampaccia} implies that $\overline p\in X_0$ and $\overline q\in X_T$. Further, in practice, observation data are typically available through measurements at various time points. Accordingly, their usual extrapolations are smooth in time. Therefore, \cref{assumption2} is reasonable since smoothness is only considered in time, whereas the data are allowed to be non-smooth with respect to the space variable.

\begin{theorem}\label{theorem well posedness}
	Let \cref{assumption3} hold. Then, for every $k\in\mathbb N$, the system \cref{iteration} admits at least one solution $(\nu_{k+1}, p_{k+1}, q_{k+1})\in\mathcal V_{ad}\times X_0\times X_T$. In particular, \cref{alg:sqp} is well-defined.
\end{theorem}

\begin{proof}
	Let $(\nu_k, p_k, q_k)\in\mathcal V_{ad}\times X_0 \times X_T$ be given for some $k\in\mathbb N_0$. By $G_k\colon L^2(\Omega)\to C^3(I, H^1_D(\Omega))$ we denote the affine-linear and continuous solution operator that maps every $\nu$ to the unique solution $p$ to
	\begin{equation*}
		\left\{\begin{aligned}
			&\nu_k\partial_t^2 p - \Delta p + \eta\partial_t p = f - (\nu - \nu_k)\partial_t^2 p_k
			&&\text{in }I\times\Omega
			\\
			&\partial_n p = 0
			&&\text{in }I\times\Gamma_N
			\\
			&p = 0
			&&\text{in }I\times\Gamma_D
			\\
			&(p, \partial_t p)(0) = (0, 0)
			&&\text{in }\Omega.
		\end{aligned}\right.
	\end{equation*}
	Note that the well-definedness of $G_k$ follows from \cref{assumption3} and \cref{lemma: abstract}.
	Making use of $G_k$, we consider the following minimization problem:
	\begin{equation}\label{optimal control problem k}
		\inf_{\nu\in\mathcal V_{ad}} J_k(\nu) \coloneqq \mathcal J(\nu,G_k\nu) - ((\nu - \nu_k)\partial_t^2(G_k\nu - p_k), q_k)_{L^2(I, L^2(\Omega))}
	\end{equation}
	To prove the existence of a minimizer to \cref{optimal control problem k}, it remains to show that $J_k\colon L^2(\Omega)\to \mathbb R$ is lower sequentially semicontinuous. The lower sequential semicontinuity of the first term is obvious since it is convex and continuous. For the second term, we notice that the embedding $C^1(I, H^1_D(\Omega))\hookrightarrow C(I, L^2(\Omega))$ is compact due to the Arzela-Ascoli theorem. Then, along with the continuity and affine-linearity of $\partial_t^2G_k\colon L^2(\Omega)\to C^1(I, H^1_D(\Omega))$, we obtain the following implication: 
	\begin{equation*}
		\nu_n\rightharpoonup\nu\quad\text{weakly in }L^2(\Omega)\quad\Rightarrow\quad(\nu_n\partial_t^2G_k\nu_n, q_k)_{L^2(I, L^2(\Omega))}\to(\nu\partial_t^2G_k\nu, q_k)_{L^2(I, L^2(\Omega))}.
	\end{equation*}
	In conclusion, $J_k$ is lower sequentially semicontinuous, and therefore \cref{optimal control problem k} admits at least one minimizer $\nu_{k+1}\in\mathcal V_{ad}$. On the other hand, the iteration system \cref{iteration} is equivalent to the condition that $J_k'(\nu)(\tilde\nu - \nu)\geq 0$ for every $\tilde\nu\in\mathcal V_{ad}$ which is nothing but the necessary optimality condition to \cref{optimal control problem k}. Therefore, \cref{iteration} admits at least one solution $(\nu_{k+1}, p_{k+1}, q_{k+1})$. Applying \cref{assumption3} and \cref{lemma: stampaccia} to the PDE systems in \cref{iteration} yields $p_{k+1}\in X_0$ and $q_{k+1}\in X_T$. In conclusion, the claim follows inductively. 
\end{proof}

\begin{assumption}\label{assumption4}
	Let \cref{assumption3} and \cref{ssc'} hold. Furthermore, suppose for every $l\in\mathbb N_0$ that
	\begin{align*}
		&\begin{aligned}
			&\|\nu_0 - \overline\nu\|_{L^2(\Omega)}\leq \varepsilon\quad\|\partial_t^l (p_0 - \overline p)\|_{L^2(I,L^\infty(\Omega))} \leq C_0\varepsilon l!^2,
			&&\|\partial_t^l (q_0 - \overline q)\|_{L^2(I,L^\infty(\Omega))} \leq C_0\varepsilon l!^2,
			\\
			&\max\left\{\|\partial_t^lp_{0}\|_{L^2(I,L^\infty(\Omega))},\|\partial_t^lq_{0}\|_{L^2(I,L^\infty(\Omega))} \right\}\leq C_0!^2,
			&&\max\left\{\|\partial_t^l\overline p\|_{L^2(I,L^\infty(\Omega))},\|\partial_t^l\overline q\|_{L^2(I,L^\infty(\Omega))} \right\}\leq \overline C l!^2,
		\end{aligned}
		\\
		&\sum_{i=1}^{m}\|\partial_t^la_i\|_{L^\infty(I,L^\infty(\Omega))}\leq C_a l!^2,\qquad\sum_{i=1}^{m}\|\partial_t^l(a_i(\overline p - p_i^{ob}))\|_{L^2(I,L^2(\Omega))}\leq C_a l!^2,
		\qquad \|\partial_t^lf\|_{L^2(I,L^2(\Omega))}\leq C_f l!^2
	\end{align*}
	and 
	\begin{equation}\label{assumption epsilon}
		\varepsilon\leq\min\Bigg\{\frac{\gamma}{4\delta}, \frac{c_1 \overline \gamma}{8 \delta}, \frac{1}{2 L(2c_1 5!^2 + 8C_0 + 2\sqrt{|\Omega|}5!^2C_0c_1)},\sqrt{\frac{8!^2\gamma^{\sqrt{2}}}{8\delta L(2c_05!^2 + \sqrt{|\Omega|}C_0(4c_03!^2 + 4C_0 +c_05!^2))}}\Bigg\}
	\end{equation}
	hold for some constants $C_f, C_a, C_0, \overline C, \gamma, \overline \gamma >0$ satisfying 
	\begin{align}\label{assumption gammas}
		&\gamma\prod_{l=1}^\infty(3l+5)!^{\sqrt{2}^{6-l}}\eqqcolon\overline\gamma<\min\left\{1,\frac{\delta}{L(4c_1 + 2\sqrt{|\Omega|}c_1^2)}, \frac{2\delta \tau}{c_L(8C_0 + 4\overline C + 3c_1C_0 + c_1\overline C)}\right\},
		\\\label{assumption c constants}
		&2\hat c C_f + \hat cC_0\frac{\overline\gamma}{\delta} \leq C_0,\qquad\qquad \hat c(2C_a + C_acT(\overline C + C_0)\frac{\overline\gamma}{\delta} + C_0\frac{\overline\gamma}{\delta})\leq C_0,
	\end{align}
	where
	\begin{equation}\label{definition delta}
		\delta \coloneqq 2L(2c_1 + 3\sqrt{|\Omega|}c_1^2),\quad c_0 \coloneqq C_0(\hat c\frac{\gamma}{2\delta}(C_acT + 1) + 1), \quad c_1 \coloneqq \hat c(2C_acT(\overline C + C_0) + 2\overline C + 4C_0)
	\end{equation}
	with $\tau, c,\hat c, L>0$ as in \cref{ssc'}, \cref{lemma: abstract}, \cref{lemma: stampaccia}, and \cref{theorem:lipschitzproperty}, respectively.
\end{assumption}

\begin{remark}
	Notice that \cref{assumption gammas}, \cref{assumption c constants}, and \cref{assumption epsilon} can always be guaranteed by choosing $\gamma$ and $\varepsilon$ sufficiently small and $C_0$ sufficiently large. Furthermore, it is straightforward to check that $\prod_{l=1}^\infty(3l+5)!^{\sqrt{2}^{6-l}}$ exists. Indeed, it holds for every $k\in\mathbb N$ that 
	\begin{align*}
		\ln\left(\prod_{l=1}^k(3l+5)!^{\sqrt{2}^{6-l}}\right) 
		&= \sum_{l=1}^k\sqrt{2}^{6-l}\ln((3l+5)!) = \sum_{l=1}^k\sqrt{2}^{6-l}\sum_{s=1}^{3l+5}\ln(s)\leq \sum_{l=1}^k\sqrt{2}^{6-l}\sum_{s=1}^{3l+5}s 
		\\
		&= \sum_{l=1}^k\sqrt{2}^{6-l}\frac{(3l+5)(3l+6)}{2} \eqqcolon \sum_{l=1}^k d_l
	\end{align*}
	and 
	\begin{equation*}
		\left|\frac{d_{l+1}}{d_l}\right| = \frac{\sqrt{2}^{5-l}\frac{(3l+8)(3l+9)}{2}}{\sqrt{2}^{6-l}\frac{(3l+5)(3l+6)}{2}} = \sqrt{2}^{-1}\frac{(3l+8)(3l+9)}{(3l+5)(3l+6)}\leq \sqrt{2}^{-1}\frac{(12+8)(12+9)}{(12+5)(12+6)}<1\quad\forall l\geq 4.
	\end{equation*}
	Thus, by the ratio test, $\{\sum_{l=1}^kd_l\}_{k\in\mathbb N}$ is convergent. Consequently, as $\exp\in C(\mathbb R)$, the limit $\prod_{l=1}^\infty(3l+5)!^{\sqrt{2}^{6-l}}$ exists.
\end{remark}

\begin{remark} 
	To justify \cref{assumption4}, we provide an example for a smooth function $f \in C^\infty(\mathbb R)$ satisfying 
	\begin{equation*}
		\partial_t^l f(0)=0 \quad \text{and} \quad \|\partial_t^l f\|_{L^2(0,1)}\leq {C_{f}} l!^2 \quad \forall l \in \mathbb N_0	
	\end{equation*}
	with a positive constant $C_f$, independent of $l$. Let $m \ge 2$ be an integer and 
	\begin{equation*}
		f_m(t) \coloneqq\left\{\begin{aligned} 
			&e^{-t^{-m}} &\text{ if } t > 0,
			\\
			&0 &\text{ if } t \leq 0.
		\end{aligned}\right.	
	\end{equation*}
	By definition, $f_m \in C^\infty(\mathbb R)$ satisfies $\partial_t^lf_m(0)=0$ for all $l \in \mathbb N_0$. Moreover, as shown in the proof of \cite[Lemma 2]{israel}, it holds that
	\begin{equation*}
		|\partial_t^l f_m (t)| \leq (8m)^l l^{(1+\frac{1}{m})l} \quad \forall t \in \mathbb R \quad \forall l \in \mathbb N.	
	\end{equation*}
	Applying the Stirling estimate 
	\begin{equation}\label{Stirling0} 
		l^l e^{-l} \leq l! \quad \forall l \in \mathbb N,
	\end{equation}
	it follows that 
	\begin{equation}\label{Stirling1} 
		|\partial_t^l f_m (t)| \leq (e^{1+ \frac{1}{m}}8m)^l l!^{1+\frac{1}{m}}\quad \forall t \in \mathbb R \quad \forall l \in \mathbb N.
	\end{equation}
	Let us choose a fixed integer ${L}_m \in \mathbb N$ such that 
	\begin{align*}
		e^{1+ \frac{1}{m}}8m \leq {L}_m^{\frac{1}{m}} e^{-{\frac{1}{m}}} \quad \Rightarrow \quad (e^{1+ \frac{1}{m}}8m)^l \leq (l^{\frac{1}{m}} e^{-{\frac{1}{m}}})^l \quad \forall l \ge {L}_m \underbrace{\Rightarrow}_{\cref{Stirling0}} \quad
		&(e^{1+ \frac{1}{m}}8m)^l \leq l!^{\frac{1}{m}} \quad \forall l \ge {L}_m 
		\\
		{\Rightarrow}\quad 
		&(e^{1+ \frac{1}{m}}8m)^l \leq C_m l!^{\frac{1}{m}} \quad \forall l \in \mathbb N
	\end{align*}
	with $C_m\coloneqq(e^{1+ \frac{1}{m}}8m)^{L_m}$. Applying the above inequality to \cref{Stirling1} finally yields
	\begin{equation*}
		|\partial_t^l f_m (t)| \leq C_m l!^{1+\frac{2}{m}} \underbrace{\le}_{m \ge 2} C_m l!^{2} \quad \forall t \in \mathbb R \quad \forall l \in \mathbb N.	
	\end{equation*}
\end{remark}

\begin{remark} 
	Let us explain our subsequent strategy and how \cref{assumption4} comes into play. In the following, we introduce two operators $S_k$ and $T_k$. The first one represents the solution operator for the perturbed system \cref{system: linpertOC2} with perturbation terms \cref{perturbations}, whereas $T_k$ is related to the solution operator for the PDE-systems in \cref{iteration}. In \cref{lemma: estimates}, we derive Lipschitz estimates for $T_k$ by invoking \cref{lemma: abstract,lemma: stampaccia} and the two conditions in \cref{assumption c constants} from \cref{assumption4}. Here, the two constants $c_0$ and $c_1$ from \cref{assumption4} serve precisely as the corresponding Lipschitz constants. Finally, with the help of $S_k, T_k$, the perturbation analysis (\cref{theorem:lipschitzproperty,lemma: firstorderrho}), and the associated Lipschitz estimates (\cref{lemma: estimates}), we prove in \cref{proposition contraction,proposition uniqueness} that \cref{iteration} admits a unique solution in $ \mathcal V_{ad}^\tau \times X_0 \times X_T$ satisfying certain properties that are vital for our convergence analysis. Our proofs for \cref{proposition contraction,proposition uniqueness} exactly rely on the use of the two smallness conditions \cref{assumption epsilon} and \cref{assumption gammas} from \cref{assumption4}.
\end{remark}

Associated with a given $(\nu_k, p_k, q_k)\in \mathcal V_{ad} \times X_0 \times X_T$ for some $k\in\mathbb N_0$, we introduce the mapping
\begin{align*}
	S_{k} \colon L^2(\Omega)\times X_0\times X_T&\to L^2(\Omega)\times X_0\times X_T,\quad (\hat\nu, \hat p, \hat q) \mapsto (\nu, p, q)
\end{align*}
that assigns to every $(\hat\nu, \hat p, \hat q)\in L^2(\Omega)\times X_0\times X_T$ the solution $(\nu, p, q)$ to 
\begin{equation}\label{system: Sk}
	\left\{\begin{aligned}
		&\overline\nu\partial_t^2 p - \Delta p + \eta\partial_t p = f - (\nu_k - \overline\nu)\partial_t^2\hat p - (\hat\nu - \nu_k)\partial_t^2p_k - (\nu - \hat\nu)\partial_t^2\overline p
		&&\text{in }I\times\Omega
		\\
		&\partial_np = 0 
		&&\text{on }I\times\Gamma_N
		\\
		&p = 0 
		&&\text{on }I\times\Gamma_D
		\\
		&(p, \partial_t p)(0) = (0, 0)
		&&\text{in }\Omega
		\\
		&\overline\nu\partial_t^2 q - \Delta q - \eta\partial_t q = \sum_{i=1}^m a_i(p - p_i^{ob}) - (\nu_k - \overline\nu)\partial_t^2\hat q - (\hat\nu - \nu_k)\partial_t^2q_k - (\nu - \hat\nu)\partial_t^2\overline q
		&&\text{in }I\times\Omega
		\\
		&\partial_nq = 0
		&&\text{on }I\times\Gamma_N
		\\
		&q = 0
		&&\text{on }I\times\Gamma_D
		\\
		&(q, \partial_t q)(T) = (0, 0)
		&&\text{in }\Omega
		\\
		&\nu\in\mathcal V_{ad}^\tau, \quad (-\int_I\partial_t^2\overline p(t) q(t) + \partial_t^2(p_k(t)-\overline p(t))\hat q(t) + \partial_t^2(\hat p(t) - p_k(t))q_k(t) 
		\\
		&\hspace{4.3cm} + \partial_t^2(p(t) - \hat p(t))\overline q(t)\d t + \lambda\nu, \tilde \nu - \nu)_{L^2(\Omega)}\geq 0
		&&\text{for all }\tilde \nu\in\mathcal V_{ad}^\tau.
	\end{aligned}\right.
\end{equation}

\begin{remark}\label{remark Sk}
	The system \cref{system: Sk} is nothing but \cref{system: linpertOC2} with the perturbation terms
	\begin{align}\label{perturbations}
		\rho^{st}
		&= - (\nu_k - \overline\nu)\partial_t^2\hat p - (\hat\nu - \nu_k)\partial_t^2p_k - (\overline\nu - \hat\nu)\partial_t^2\overline p
		&&\in H^1(I, L^2(\Omega))
		\\\notag
		\rho^{adj}
		&= - (\nu_k - \overline\nu)\partial_t^2\hat q - (\hat\nu - \nu_k)\partial_t^2q_k - (\overline\nu - \hat\nu)\partial_t^2\overline q
		&&\in H^1(I, L^2(\Omega))
		\\\notag
		\rho^{VI}
		&=\int_I\partial_t^2(p_k(t)-\overline p(t))\hat q(t) + \partial_t^2(\hat p(t) - p_k(t))q_k(t) + \partial_t^2(\overline p(t) - \hat p(t))\overline q(t)\d t
		&&\in L^2(\Omega)
	\end{align}
	satisfying $\rho^{st}(0) = \rho^{adj}(T) = 0$ such that the well-definedness of $S_{k}$ follows by \cref{proposition: well-definedness} and \cref{lemma: stampaccia}.
\end{remark}

We aim to show that $S_{k}$ under certain assumptions admits a unique fixed point. According to \cref{iteration}, every fixed point to $S_k$ exactly solves the iteration in \cref{alg:sqp} with $\mathcal V_{ad}^\tau$ instead of $\mathcal V_{ad}$. Unfortunately, due to the nature of the hyperbolic PDEs and the second-order time derivatives in the source terms, $S_k$ cannot be defined as a self-map in an appropriate Banach space (see \cref{lemma: abstract}). As a consequence, the contraction principle is not applicable directly to $S_k$. As mentioned in the introduction, we establish a suitable self-map to overcome this issue. First, we define for every $k\in\mathbb N_0$ the mapping
\begin{equation}\label{def Tk}
	T_{k}\colon L^2(\Omega) \to L^2(\Omega)\times X_0\times X_T, \quad \hat\nu \mapsto(\hat\nu, \hat p, \hat q)
\end{equation}
where $\hat p, \hat q$ solve
\begin{align}\label{system: statetk}
	&\left\{\begin{aligned}
		&\nu_k\partial_t^2\hat p - \Delta \hat p + \eta\partial_t \hat p = f - (\hat \nu - \nu_k)\partial_t^2 p_k
		&&\text{in }I\times\Omega
		\\
		&\partial_n\hat p = 0
		&&\text{on }I\times\Gamma_N
		\\
		&\hat p = 0
		&&\text{on }I\times\Gamma_D
		\\
		&(\hat p, \partial_t \hat p)(0) = (0,0)
		&&\text{in }\Omega
	\end{aligned}\right.
\end{align}
and
\begin{align}\label{system: adjointtk}
	&\left\{\begin{aligned}
		&\nu_k\partial_t^2\hat q - \Delta \hat q - \eta\partial_t \hat q = \sum_{i=1}^m a_i(\hat p - p_i^{ob}) - (\hat \nu - \nu_k)\partial_t^2 q_k
		&&\text{in }I\times\Omega
		\\
		&\partial_n\hat q = 0
		&&\text{on }I\times\Gamma_N
		\\
		&\hat q = 0
		&&\text{on }I\times\Gamma_D
		\\
		&(\hat q, \partial_t \hat q)(T) = (0,0)
		&&\text{in }\Omega.
	\end{aligned}\right.
\end{align}
Note that \cref{system: statetk} and \cref{system: adjointtk} correspond to the PDE-systems in \cref{iteration}.
The well-definedness of $T_{k}$ follows from \cref{lemma: stampaccia}. Then, the desired self-mapping operator reads
\begin{equation}\label{restriction}
	(I_{\nu}\circ S_{k}\circ T_{k})\colon L^2(\Omega)\to L^2(\Omega)\quad\text{with}\quad I_{\nu}\colon(\nu,p,q)\mapsto\nu.
\end{equation}
As we will see later, the operator \cref{restriction} is constructed suitably such that, under certain assumptions, it admits a unique fixed point $\nu_{k+1}$. Furthermore, $T_{k}(\nu_{k+1})$ is the fixed point of $S_{k}$ and solves the iteration \cref{iteration}. To prove these results, let us start with the following auxiliary lemmata:

\begin{lemma}\label{lemma sequences}
	Let $\gamma>0$ such that
	\begin{equation}\label{assumption gamma}
		\gamma\prod_{l=1}^\infty(3l+5)!^{\sqrt{2}^{6-l}}\eqqcolon\overline\gamma\in(0,1).
	\end{equation}
	Then, the sequence $\{b_k\}_{k\in\mathbb N_0}\subset\mathbb R_+$ defined by 
	\begin{equation*}
		b_0\coloneqq \gamma,\quad b_k\coloneqq \prod_{l=1}^{k}(3l+5)!^{\sqrt{2}^{2+k-l}}\gamma^{\sqrt{2}^k}\quad\forall k\in\mathbb N
	\end{equation*}
	is decreasing and converges R-superlinearly to $0$. Furthermore, it holds that 
	\begin{align}\label{bk overline gamma}
		b_k\leq\overline\gamma^{\sqrt{2}^k}&\leq\overline\gamma
		&&\forall k\in\mathbb N_0
		\\\label{ieq gamma bar}
		(3k+5)!^4b_{k-1}&\leq \overline\gamma
		&&\forall k\in\mathbb N.
	\end{align}
	Additionally, suppose that $\{x_n\}_{n\in\mathbb N_0}\subset\mathbb R_+$ satisfies for some $k\in\mathbb N$ and some $\delta>0$ that
	\begin{equation}\label{assumption x}
		x_{k-1}\leq \frac{1}{4\delta}b_{k-1},\quad x_k\leq \frac{1}{4\delta}b_k,\quad x_{k+1}\leq \delta(3k+5)!^4(x_k + x_{k-1})^2.
	\end{equation}
	Then
	\begin{equation}\label{claim xk}
		x_{k+1}\leq \frac{1}{4\delta}b_{k+1}.
	\end{equation}
\end{lemma}

\begin{proof}
	First, we notice that \cref{assumption gamma} implies that $b_0=\gamma\leq\overline\gamma$. Furthermore, for every $k\in\mathbb N$, we have that
	\begin{equation*}
		b_k = \prod_{l=1}^{k}(3l+5)!^{\sqrt{2}^{2+k-l}}\gamma^{\sqrt{2}^k} = \left(\prod_{l=1}^{k}(3l+5)!^{\sqrt{2}^{2-l}}\gamma\right)^{\sqrt{2}^k}\underbrace{\leq}_{\cref{assumption gamma}}\overline\gamma^{\sqrt{2}^k}\quad\forall k\in\mathbb N,
	\end{equation*}
	and therefore \cref{bk overline gamma} is valid. Since $\{\overline\gamma^{\sqrt{2}^k}\}_{k\in\mathbb N_0}$ converges (Q-)superlinearly to $0$, the sequence $\{b_k\}_{k\in\mathbb N_0}$ converges R-superlinearly to $0$. To prove the monotonicity, notice that $b_1 = 8!^2\gamma^{\sqrt{2}} = 8!^2b_0^{\sqrt{2}-1}b_0$ and 
	\begin{equation*}
		b_k = (3k+5)!^2\left(\prod_{l=1}^{k-1}(3l+5)!^{\sqrt{2}^{1+k-l}}\gamma^{\sqrt{2}^{k-1}}\right)^{\sqrt{2}} = (3k+5)!^2b_{k-1}^{\sqrt{2}} = (3k+5)!^2b_{k-1}^{\sqrt{2}-1}b_{k-1}\quad\forall k\geq 2.
	\end{equation*}
	Thus, $\{b_k\}_{k\in\mathbb N_0}$ is decreasing if we can show that $(3k+5)!^2b_{k-1}^{\sqrt{2}-1}\in (0,1)$ for all $k\in\mathbb N$. Indeed, this holds since
	\begin{equation*}
		0<\left(8!^2b_0^{\sqrt{2}-1}\right)^{\sqrt{2}+1}\leq 8!^{\sqrt{2}^5}b_0 = 8!^{\sqrt{2}^5}\gamma\underbrace{\leq}_{\cref{assumption gamma}}\overline\gamma<1 
	\end{equation*}
	and
	\begin{align}\label{bk estimate}
		0&<\left((3k+5)!^2b_{k-1}^{\sqrt{2}-1}\right)^{\sqrt{2}+1}\leq (3k+5)!^{\sqrt{2}^5}b_{k-1} = (3k+5)!^{\sqrt{2}^5}\prod_{l=1}^{k-1}(3l+5)!^{\sqrt{2}^{1+k-l}}\gamma^{\sqrt{2}^{k-1}}
		\\\notag
		&\leq \prod_{l=1}^{k}(3l+5)!^{\sqrt{2}^{5+k-l}}\gamma^{\sqrt{2}^{k-1}} = \left(\prod_{l=1}^{k}(3l+5)!^{\sqrt{2}^{6-l}}\gamma\right)^{\sqrt{2}^{k-1}}\underbrace{\leq}_{\cref{assumption gamma}}\overline\gamma^{\sqrt{2}^{k-1}}< 1\quad\forall k\geq 2.
	\end{align}
	The claim \cref{ieq gamma bar} for $k=1$ follows immediately from \cref{assumption gamma}. For $k\geq 2$, the claim \cref{ieq gamma bar} is obtained as follows:
	\begin{equation*}
		(3k+5)!^4b_{k-1}\leq(3k+5)!^{\sqrt{2}^5}b_{k-1}\underbrace{\leq}_{\cref{bk estimate}}\overline\gamma^{\sqrt{2}^{k-1}}\leq\overline\gamma\quad\forall k\geq 2.
	\end{equation*}
	Now, suppose that $\{x_n\}_{n\in\mathbb N_0}\subset\mathbb R_+$ satisfies \cref{assumption x} for some $k\in\mathbb N$ and some $\delta>0$. If $k=1$, using the monotonicity $b_1\leq b_0$, we obtain that
	\begin{align*}
		x_2\underbrace{\leq}_{\cref{assumption x}} \delta8!^4(x_1 + x_0)^2\leq \frac{1}{16\delta}8!^4(b_1 + b_0)^2\leq \frac{1}{4\delta}8!^4b_0^2=\frac{1}{4\delta}8!^4\gamma^2\leq \frac{1}{4\delta}8!^{\sqrt{2}^3}11!^2\gamma^2 = \frac{1}{4\delta}b_2.
	\end{align*}
	If $k\geq 2$, again using the monotonicity $b_k\leq b_{k-1}$, we obtain that
	\begin{align*}
		x_{k+1}
		&\underbrace{\leq}_{\cref{assumption x}} \delta(3k+5)!^4(x_k + x_{k-1})^2\leq \frac{1}{16\delta}(3k+5)!^4(b_k + b_{k-1})^2\leq \frac{1}{4\delta}(3k+5)!^4b_{k-1}^2
		\\
		&= \frac{1}{4\delta}(3k+5)!^4\prod_{l=1}^{k-1}(3l+5)!^{\sqrt{2}^{3+k-l}}\gamma^{\sqrt{2}^{k+1}}\leq \frac{1}{4\delta}\prod_{l=1}^{k+1}(3l+5)!^{\sqrt{2}^{3+k-l}}\gamma^{\sqrt{2}^{k+1}}= \frac{1}{4\delta}b_{k+1}.
	\end{align*}
	This completes the proof. 
\end{proof}

\begin{lemma}\label{lemma: estimates}
	Let \cref{assumption4} hold. Then, for $\nu,\tilde\nu\in\mathcal V_{ad}$, $(\overline\nu,\check p, \check q)\coloneqq T_0(\overline\nu)$, $(\nu,p, q)\coloneqq T_0(\nu)$, and $(\tilde\nu,\tilde p, \tilde q)\coloneqq T_0(\tilde\nu)$, it holds for all $l\in\mathbb N_0$ that 
	\begin{align}
		\label{ieq 1}
		&\|\partial_t^l (\check p - p_0)\|_{L^2(I,L^\infty(\Omega))}\leq c_0(l+3)!^2\varepsilon
		\\\label{ieq 1b}
		&\|\partial_t^l (\check q - q_0)\|_{L^2(I,L^\infty(\Omega))} \leq c_0(l+3)!^2\varepsilon
		\\\label{ieq 3}
		&\|\partial_t^l ( p - \tilde p)\|_{L^2(I,L^\infty(\Omega))}\leq c_1(l+3)!^2\|\nu - \tilde\nu\|_{L^2(\Omega)}
		\\\label{ieq 3b}
		&\|\partial_t^l (q - \tilde q)\|_{L^2(I,L^\infty(\Omega))}\leq c_1(l+3)!^2\|\nu - \tilde\nu\|_{L^2(\Omega)}
	\end{align} 
	with $c_0, c_1>0$ as in \cref{assumption4}. Let additionally $\nu_k$ and $(\nu_{k-1}, p_{k-1}, q_{k-1})$ for some $k\in\mathbb N$ be given such that
	\begin{equation}\label{assumptionk}\tag{$\text{A}_k$}
		\left\{\begin{aligned}
			&\nu_k,\nu_{k-1}\in \mathcal V_{ad}, p_{k-1}\in X_0, q_{k-1}\in X_T
			\\
			&\max\{\|\partial_t^lp_{k-1}\|_{L^2(I,L^\infty(\Omega))}, \|\partial_t^lq_{k-1}\|_{L^2(I,L^\infty(\Omega))}\} \leq C_0(l+3k-3)!^2 \quad\forall l\in\mathbb N_0
			\\
			&\|\nu_k - \overline\nu\|_{L^2(\Omega)}\leq\frac{1}{4\delta}b_k,\quad \|\nu_{k-1} - \overline\nu\|_{L^2(\Omega)}\leq\frac{1}{4\delta}b_{k-1}
		\end{aligned}\right.
	\end{equation}
	with $C_0,\delta>0$ as in \cref{assumption4} and $b_k, b_{k-1}$ as in \cref{lemma sequences}. Then, for $\nu,\tilde\nu\in\mathcal V_{ad}$, $(\nu,p, q)\coloneqq T_{k}(\nu)$, $(\tilde\nu,\tilde p, \tilde q)\coloneqq T_{k}(\tilde\nu)$, and $p_k\in X_0, q_k\in X_T$ being the unique solutions to
	\begin{align}\label{iterationstate}
		&\begin{aligned}
			&\left\{\begin{aligned}
				&\nu_{k-1}\partial_t^2 p_k -\Delta p_k + \eta\partial_t p_k = f -(\nu_k - \nu_{k-1})\partial_t^2 p_{k-1} 
				&&\text{in } I\times \Omega
				\\
				&\partial_n p_k = 0 
				&&\text{on }I\times\Gamma_N
				\\
				&p_k = 0 
				&&\text{on }I\times\Gamma_D	
				\\
				&(p_k, \partial_t p_k)(0) = (0, 0) 
				&&\text{in } \Omega
			\end{aligned}\right.
		\end{aligned}
		\\\label{iterationadjoint}
		&\begin{aligned}
			&\left\{\begin{aligned}
				&\nu_{k-1}\partial_t^2 q_k -\Delta q_k - \eta\partial_t q_k = \sum_{i=1}^m a_i(p_k - p_i^{ob}) -(\nu_k - \nu_{k-1})\partial_t^2 q_{k-1} 
				&&\text{in } I\times \Omega
				\\
				&\partial_n q_k = 0 
				&&\text{on }I\times\Gamma_N
				\\
				&q_k = 0 
				&&\text{on }I\times\Gamma_D	
				\\
				&(q_k, \partial_t q_k)(T) = (0, 0) 
				&&\text{in } \Omega,
			\end{aligned}\right.
		\end{aligned}
	\end{align}
	it holds that 
	\begin{equation}\label{pk claim}
		\max\{\|\partial_t^lp_k\|_{L^2(I,L^\infty(\Omega))},\|\partial_t^lq_k\|_{L^2(I,L^\infty(\Omega))} \} \leq C_0(l+3k)!^2 \quad\forall l\in\mathbb N_0
	\end{equation}
	and for all $l\in\mathbb N_0$ that 
	\begin{align}
		\label{ieq 5}
		&\|\partial_t^l ( p - \tilde p)\|_{L^2(I,L^\infty(\Omega))}\leq c_1(l+3k+3)!^2\|\nu - \tilde\nu\|_{L^2(\Omega)}
		\\\label{ieq 5b}
		&\|\partial_t^l (q - \tilde q)\|_{L^2(I,L^\infty(\Omega))}\leq c_1(l+3k+3)!^2\|\nu - \tilde\nu\|_{L^2(\Omega)}
		\\\label{ieq 7}
		&\|\partial_t^l (\overline p - p_k)\|_{L^2(I,L^\infty(\Omega))}\leq c_1(l+3k)!^2(\|\overline\nu - \nu_k\|_{L^2(\Omega)} + \|\overline\nu - \nu_{k-1}\|_{L^2(\Omega)})
		\\\label{ieq 8}
		&\|\partial_t^l (\overline q - q_k)\|_{L^2(I,L^\infty(\Omega))}\leq c_1(l+3k)!^2(\|\overline\nu - \nu_k\|_{L^2(\Omega)} + \|\overline\nu - \nu_{k-1}\|_{L^2(\Omega)})
		\\\label{ieq 9}
		&\|\partial_t^l (p - p_k)\|_{L^2(I,L^\infty(\Omega))}\leq c_1(l+3k+3)!^2(\|\nu - \nu_k\|_{L^2(\Omega)} + \|\nu - \nu_{k-1}\|_{L^2(\Omega)})
		\\\label{ieq 10}
		&\|\partial_t^l (q - q_k)\|_{L^2(I,L^\infty(\Omega))}\leq c_1(l+3k+3)! ^2(\|\nu - \nu_k\|_{L^2(\Omega)} + \|\nu - \nu_{k-1}\|_{L^2(\Omega)}).
	\end{align}
\end{lemma}

\begin{proof}
	Considering the corresponding PDEs for the quantities $p,\tilde p, \overline p, p_k, q,\tilde q, \overline q$, and $q_k$ (see \cref{first order optimality,system: statetk,system: adjointtk,iterationstate,iterationadjoint}), the estimates follow from \cref{lemma: abstract}, \cref{lemma: stampaccia} along with the estimates from \cref{assumption4}. We refer to the \cref{subsection proof lemma estimates} for the detailed proof.
\end{proof}

\begin{proposition}
	\label{proposition contraction} Let \cref{assumption4} be satisfied and $\{b_k\}_{k\in\mathbb N_0}$ as in \cref{lemma sequences}. Then, the mapping $(I_{\nu}\circ S_{0}\circ T_{0})\colon L^2(\Omega) \to L^2(\Omega)$ associated with $(\nu_0, p_0, q_0)$ is a contraction and admits a unique fixed point $\nu_1\in\mathcal V_{ad}^\tau$ satisfying
	\begin{equation}\label{ieq: quadratic1}
		\|\nu_1 - \overline\nu\|_{L^2(\Omega)}\leq \frac{1}{4\delta}b_1,
	\end{equation}
	where $\overline\gamma$ and $\delta$ are as in \cref{assumption4}. Let additionally $\nu_k$ and $(\nu_{k-1},p_{k-1},q_{k-1})$ satisfy \cref{assumptionk} for some $k\in\mathbb N$. Then, the mapping $(I_{\nu}\circ S_{k}\circ T_{k})\colon L^2(\Omega) \to L^2(\Omega)$ associated with $(\nu_k, p_k, q_k)$, with $p_k$ and $q_k$ being the unique solutions to \cref{iterationstate} and \cref{iterationadjoint}, is a contraction and admits a unique fixed point $\nu_{k+1}\in\mathcal V_{ad}^\tau$ satisfying
	\begin{align}\label{ieq: quadratic2}
		\|\nu_{k+1} - \overline\nu\|_{L^2(\Omega)}
		&\leq \delta(3k+5)!^4(\|\overline\nu - \nu_k\|_{L^2(\Omega)} + \|\overline\nu - \nu_{k-1}\|_{L^2(\Omega)})^2
		\\\label{claim nu k+1}
		\|\nu_{k+1}- \overline\nu\|_{L^2(\Omega)}
		&\leq \frac{1}{4\delta}b_{k+1}.
	\end{align}
\end{proposition}

\begin{remark}
	Notice that the positive constant $\delta$ appearing in \cref{ieq: quadratic1}-\cref{claim nu k+1} is independent of $k$. More precisely, it depends only on the given data as defined in \cref{assumption4}.
\end{remark}

\begin{proof}
	Let $\nu, \tilde\nu\in L^2(\Omega)$ and $k\in\mathbb N_0$. Further, if $k\in\mathbb N$, suppose that $\nu_k$ and $(\nu_{k-1},p_{k-1},q_{k-1})$ satisfy \cref{assumptionk} and $p_k$, $q_k$ denote the unique solutions to \cref{iterationstate} and \cref{iterationadjoint}, respectively. According to \cref{def Tk}, we may write $(\nu,p, q)\coloneqq T_{k}(\nu)$ and $(\tilde\nu, \tilde p, \tilde q)\coloneqq T_{k}(\tilde\nu)$. As pointed out in \cref{remark Sk}, $(S_{k}\circ T_{k})(\nu)$ and $(S_{k}\circ T_{k})(\tilde\nu)$ solve \cref{system: linpertOC2} with the perturbation terms \cref{perturbations} for $(\hat\nu, \hat p, \hat q) = (\nu,p, q)$ and $(\hat\nu, \hat p, \hat q) = (\tilde\nu, \tilde p, \tilde q)$, respectively. Then, by \cref{theorem:lipschitzproperty}, it holds that
	\begin{align}\label{estimate}
		&\|(I_{\nu}\circ S_{k}\circ T_{k})(\nu) - (I_{\nu}\circ S_{k}\circ T_{k})(\tilde\nu)\|_{L^2(\Omega)} 
		\\\notag
		&\leq L(\|(\nu_k - \overline\nu)\partial_t^2(p - \tilde p) + (\nu - \tilde\nu)\partial_t^2(p_k - \overline p)\|_{L^2(I, L^2(\Omega))} + \|(\nu_k - \overline\nu)\partial_t^2(q - \tilde q) + (\nu - \tilde\nu)\partial_t^2(q_k - \overline q)\|_{L^2(I, L^2(\Omega))}
		\\\notag
		&\hspace{.87cm}+ \|\int_I\partial_t^2(p_k(t) - \overline p(t))(q(t) - \tilde q(t)) + \partial_t^2(p(t) - \tilde p(t))(q_k(t) - \overline q(t))\d t\|_{L^2(\Omega)})
		\\\notag
		&\leq L(\|\nu_k - \overline\nu\|_{L^2(\Omega)}\|\partial_t^2(p - \tilde p)\|_{L^2(I, L^\infty(\Omega))} + \|\nu - \tilde\nu\|_{L^2(\Omega)}\|\partial_t^2(p_k - \overline p)\|_{L^2(I, L^\infty(\Omega))}
		\\\notag
		&\hspace{.87cm} + \|\nu_k - \overline\nu\|_{L^2(\Omega)}\|\partial_t^2(q - \tilde q)\|_{L^2(I, L^\infty(\Omega))} + \|\nu - \tilde\nu\|_{L^2(\Omega)}\|\partial_t^2(q_k - \overline q)\|_{L^2(I, L^\infty(\Omega))}
		\\\notag
		&\hspace{.87cm}+ \|\partial_t^2(p_k - \overline p)\|_{L^2(I, L^\infty(\Omega))}\|q - \tilde q\|_{L^2(I, L^2(\Omega))} + \|\partial_t^2(p - \tilde p)\|_{L^2(I, L^2(\Omega))}\|q_k - \overline q\|_{L^2(I, L^\infty(\Omega))}).
	\end{align}
	According to \cref{assumption4} and \cref{lemma: estimates} (see \cref{ieq 3,ieq 3b}), the above inequality implies for $k=0$ that 
	\begin{align*}
		\|(I_{\nu}\circ S_0\circ T_0)(\nu) - (I_{\nu}\circ S_0\circ T_0)(\tilde\nu)\|_{L^2(\Omega)} 
		&\leq L(2c_15! ^2 + 8C_0 + 2\sqrt{|\Omega|}5! ^2 C_0c_1)\varepsilon\|\nu - \tilde\nu\|_{L^2(\Omega)}
		\\
		&\hspace{-.14cm}\underbrace{\leq}_{\cref{assumption epsilon}}\frac{1}{2}\|\nu - \tilde\nu\|_{L^2(\Omega)}.
	\end{align*}
	Similarly, for $k\in\mathbb N$, we obtain, using the monotonicity $b_k\leq b_{k-1}$ (see \cref{lemma sequences}), that 
	\begin{align*}
		&\|(I_{\nu}\circ S_{k}\circ T_{k})(\nu) - (I_{\nu}\circ S_{k}\circ T_{k})(\tilde\nu)\|_{L^2(\Omega)} 
		\\
		&\hspace{-.3cm}\underbrace{\leq}_{\cref{ieq 5}-\cref{ieq 8}}L(4c_1+2\sqrt{|\Omega|}c_1^2)(3k+5)!^4(\|\overline\nu - \nu_k\|_{L^2(\Omega)} + \|\overline\nu - \nu_{k-1}\|_{L^2(\Omega)})\|\nu - \tilde\nu\|_{L^2(\Omega)}
		\\
		&\underbrace{\leq}_{\cref{assumptionk}}\frac{1}{2\delta}L(4c_1+2\sqrt{|\Omega|}c_1^2)(3k+5)!^4b_{k-1}\|\nu - \tilde\nu\|_{L^2(\Omega)}\underbrace{\leq}_{\cref{ieq gamma bar}} \frac{1}{2\delta}L(4c_1+2\sqrt{|\Omega|}c_1^2) \overline\gamma\|\nu - \tilde\nu\|_{L^2(\Omega)}\underbrace{<}_{\cref{assumption gammas}} \frac{1}{2}\|\nu - \tilde\nu\|_{L^2(\Omega)}.
	\end{align*}
	Therefore, the mapping $(I_{\nu}\circ S_{k}\circ T_{k})\colon L^2(\Omega)\to L^2(\Omega)$ is a contraction and consequently admits a unique fixed point $\nu_{k+1}\in L^2(\Omega)$ due to Banach's fixed point theorem. Also, according to \cref{system: Sk}, $I_\nu\circ S_k$ maps into $\mathcal V_{ad}^\tau$ such that $\nu_{k+1}\in \mathcal V_{ad}^\tau$. Now, let us prove \cref{ieq: quadratic1} and \cref{ieq: quadratic2}. From the above contraction property, it holds that
	\begin{equation*}
		\frac{1}{2}\|\nu_{k+1} - \overline\nu\|_{L^2(\Omega)} + \|(I_{\nu}\circ S_{k}\circ T_{k})(\nu_{k+1}) - (I_{\nu}\circ S_{k}\circ T_{k})(\overline\nu)\|_{L^2(\Omega)} \leq \|\nu_{k+1} - \overline\nu\|_{L^2(\Omega)}.
	\end{equation*}
	Thus, since $\nu_{k+1} = (I_{\nu}\circ S_{k}\circ T_{k})\nu_{k+1}$, it follows that
	\begin{equation}\label{ieq: contr1}
		\frac{1}{2}\|\nu_{k+1} - \overline\nu\|_{L^2(\Omega)} \leq \|\nu_{k+1} - \overline\nu\|_{L^2(\Omega)} - \|(I_{\nu}\circ S_{k}\circ T_{k})(\nu_{k+1}) - (I_{\nu}\circ S_{k}\circ T_{k})(\overline\nu)\|_{L^2(\Omega)} \leq \|(I_{\nu}\circ S_{k}\circ T_{k})(\overline\nu) - \overline\nu\|_{L^2(\Omega)}.
	\end{equation}
	Furthermore, we set $(\overline\nu, \check p, \check q)\coloneqq T_k(\overline\nu)$. As above, according to \cref{remark Sk}, $S_k(T_k(\overline\nu))$ solves \cref{system: linpertOC2} with the perturbation terms \cref{perturbations} with $(\hat\nu, \hat p, \hat q) = (\overline\nu, \check p, \check q)$, and $(\overline\nu, \overline p, \overline q)$ solves \cref{system: linpertOC2} with the perturbation $(\rho^{st}, \rho^{adj}, \rho^{VI}) = (0,0,0)$. Thus, \cref{theorem:lipschitzproperty} implies that
	\begin{align}\label{ieq: contr2}
		&\|(I_{\nu}\circ S_{k}\circ T_{k})(\overline\nu) - \overline\nu\|_{L^2(\Omega)}
		\\\notag
		&\leq L(\|(\nu_k - \overline\nu)\partial_t^2(\check p - p_k)\|_{L^2(I, L^2(\Omega))} + \|(\nu_k - \overline\nu)\partial_t^2(\check q - q_k)\|_{L^2(I, L^2(\Omega))}
		\\\notag
		&\quad + \left\|\int_I\partial_t^2(p_k(t)-\overline p(t))\check q(t) + \partial_t^2(\check p(t) - p_k(t))q_k(t) + \partial_t^2(\overline p(t) - \check p(t))\overline q(t)\d t\right\|_{L^2(\Omega)})
		\\\notag
		&\leq L(\|\nu_k - \overline\nu\|_{L^2(\Omega)}\|\partial_t^2(\check p - p_k)\|_{L^2(I, L^\infty(\Omega))} + \|\nu_k - \overline\nu\|_{L^2(\Omega)}\|\partial_t^2(\check q - q_k)\|_{L^2(I, L^\infty(\Omega))}
		\\\notag
		&\quad + \|\partial_t^2(p_k-\overline p)\|_{L^2(I, L^\infty(\Omega))}\|\check q - \overline q\|_{L^2(I, L^2(\Omega))} + \|\partial_t^2(\check p - p_k)\|_{L^2(I, L^\infty(\Omega))}\|q_k - \overline q\|_{L^2(I, L^2(\Omega))})
		\\\notag
		&\leq L(\|\nu_k - \overline\nu\|_{L^2(\Omega)}\|\partial_t^2(\check p - p_k)\|_{L^2(I, L^\infty(\Omega))} + \|\nu_k - \overline\nu\|_{L^2(\Omega)}\|\partial_t^2(\check q - q_k)\|_{L^2(I, L^\infty(\Omega))}
		\\\notag
		&\quad + \|\partial_t^2(p_k-\overline p)\|_{L^2(I, L^\infty(\Omega))}\|\check q - q_k\|_{L^2(I, L^2(\Omega))} + \|\partial_t^2(p_k-\overline p)\|_{L^2(I, L^\infty(\Omega))}\|q_k - \overline q\|_{L^2(I, L^2(\Omega))} 
		\\\notag
		&\quad + \|\partial_t^2(\check p - p_k)\|_{L^2(I, L^\infty(\Omega))}\|q_k - \overline q\|_{L^2(I, L^2(\Omega))}).
	\end{align} 
	For $k = 0$, applying \cref{ieq: contr2} to \cref{ieq: contr1} and making use of \cref{lemma: estimates} and \cref{assumption4} yield that
	\begin{equation*}
		\|\nu_1 - \overline\nu\|_{L^2(\Omega)}\leq 2L(2c_05! ^2 + \sqrt{|\Omega|}C_0(4c_03! ^2 + 4C_0 +c_05! ^2))\varepsilon^2\underbrace{\leq}_{\cref{assumption epsilon}}\frac{1}{4\delta}8!^2\gamma^{\sqrt{2}}= \frac{1}{4\delta}b_1.
	\end{equation*}
	Analogously, for $k\in\mathbb N$, \cref{lemma: estimates} (see \cref{ieq 7}-\cref{ieq 10} with $(\nu,p,q)=(\overline \nu, \check p, \check q)$) implies that
	\begin{equation*}
		\|\nu_{k+1} - \overline\nu\|_{L^2(\Omega)}\underbrace{\leq}_{\cref{ieq 7}-\cref{ieq 10}}\underbrace{2L(2c_1 + 3\sqrt{|\Omega|}c_1^2)}_{=\delta}(3k+5)!^4(\|\overline\nu - \nu_k\|_{L^2(\Omega)} + \|\overline\nu - \nu_{k-1}\|_{L^2(\Omega)})^2.
	\end{equation*} 
	In conclusion, \cref{ieq: quadratic1} and \cref{ieq: quadratic2} are valid. Finally, due to \cref{assumptionk} and \cref{ieq: quadratic2} and since $\|\nu_0 - \overline\nu\|_{L^2(\Omega)}\leq\varepsilon\leq\frac{\gamma}{4\delta}$ (see \cref{assumption4}), we may apply \cref{lemma sequences} with $x_k\coloneqq\|\nu_k - \overline\nu\|_{L^2(\Omega)}$ to obtain 
	\cref{claim nu k+1}.
\end{proof}

Under \cref{assumption3}, we know that \cref{alg:sqp} is well-defined (see \cref{theorem well posedness}), but the iteration step \cref{iteration} may have multiple possible solutions in $\mathcal V_{ad}\times X_0\times X_T$. In the following, we prove that under \cref{assumption4} and \cref{assumptionk}, the solution to the iteration step \cref{iteration} is unique in $\mathcal V_{ad}^\tau\times X_0\times X_T$ that is precisely given by the unique fixed point from \cref{proposition contraction}.

\begin{proposition}\label{proposition uniqueness}
	Let \cref{assumption4} be satisfied and $\nu_1\in\mathcal V_{ad}^\tau$ denote the unique fixed point of $(I_{\nu}\circ S_{0}\circ T_{0})$ associated with $(\nu_0, p_0, q_0)$. Then, $(\nu_1, p_1, q_1)\coloneqq T_0(\nu_1)$ is the unique solution to the iteration \cref{iteration} for $k=0$ in $\mathcal V_{ad}^\tau\times X_0\times X_T$. Let additionally $\nu_k$ and $(\nu_{k-1},p_{k-1},q_{k-1})$ satisfy \cref{assumptionk} for some $k\in\mathbb N$ and $\nu_{k+1}\in\mathcal V_{ad}$ denote the unique fixed point of $(I_{\nu}\circ S_{k}\circ T_{k})$ associated with $(\nu_k, p_k, q_k)$, with $p_k$ and $q_k$ being the unique solutions to \cref{iterationstate} and \cref{iterationadjoint}. Then, $(\nu_{k+1}, p_{k+1}, q_{k+1})\coloneqq T_k(\nu_{k+1})$ is the unique solution to \cref{iteration} in $\mathcal V_{ad}^\tau\times X_0\times X_T$.
\end{proposition}

\begin{proof}
	Let $k$ be either zero or as above. Let $k\in\mathbb N_0$. Since $\nu_{k+1}$ is a fixed-point of $(I_{\nu}\circ S_{k}\circ T_{k})$, it holds by the definition of $S_k$ (see \cref{system: Sk} for $(\hat\nu, \hat p, \hat q) = (\nu_{k+1}, p_{k+1}, q_{k+1})$) that $S_{k}(\nu_{k+1}, p_{k+1}, q_{k+1}) = (\nu_{k+1}, p, q)$ where $p, q$ are the unique solutions to 
	\begin{align}\label{system pk+1}
		&\left\{\begin{aligned}
			&\overline\nu\partial_t^2 p - \Delta p + \eta\partial_t p = f - (\nu_k - \overline\nu)\partial_t^2 p_{k+1} - (\nu_{k+1} - \nu_k)\partial_t^2p_k
			&&\text{in }I\times\Omega
			\\
			&\partial_np = 0
			&&\text{on }I\times\Gamma_N
			\\
			&p = 0
			&&\text{on }I\times\Gamma_D
			\\
			&(p, \partial_t p)(0) = (0, 0)
			&&\text{in }\Omega		
		\end{aligned}\right.
	\end{align}
	and
	\begin{align}\label{system qk+1}
		&\left\{\begin{aligned}
			&\overline \nu\partial_t^2q - \Delta q - \eta\partial_t q = \sum_{i=1}^m a_i(p - p_i^{ob}) - (\nu_k - \overline\nu)\partial_t^2 q_{k+1} - (\nu_{k+1} - \nu_k)\partial_t^2q_k
			&&\text{in }I\times\Omega
			\\
			&\partial_nq = 0
			&&\text{on }I\times\Gamma_N
			\\
			&q = 0
			&&\text{on }I\times\Gamma_D
			\\
			&(q, \partial_t q)(T) = (0,0)
			&&\text{in }\Omega.
		\end{aligned}\right.
	\end{align}
	Furthermore, according to the definition of $T_{k}$ (see \cref{def Tk} for $(\hat\nu, \hat p, \hat q) = (\nu_{k+1}, p_{k+1}, q_{k+1})$), $p_{k+1}$ and $q_{k+1}$, respectively, solve the same systems \cref{system pk+1} and \cref{system qk+1}. Therefore, we obtain $p = p_{k+1}$ and $q = q_{k+1}$, and consequently $T_{k}(\nu_{k+1})$ is a fixed point of $S_{k}$ and satisfies the PDEs in \cref{iteration}. Let us prove that $T_{k}(\nu_{k+1})$ satisfies the variational inequality in \cref{iteration}. Note that the assumptions of \cref{lemma: estimates} are fulfilled. On the other hand, in view of \cref{remark Sk}, the fixed point $(\nu_{k+1}, p_{k+1}, q_{k+1})$ of $S_k$ solves \cref{system: linpertOC2} with the perturbation terms
	\begin{align}\label{perturbations iteration}
		\rho^{st}
		&= - (\nu_k - \overline\nu)\partial_t^2 p_{k+1} - (\nu_{k+1} - \nu_k)\partial_t^2p_k - (\overline\nu - \nu_{k+1})\partial_t^2\overline p 
		&&\in H^1(I, L^2(\Omega))
		\\\notag
		\rho^{adj}
		&= - (\nu_k - \overline\nu)\partial_t^2q_{k+1} - (\nu_{k+1} - \nu_k)\partial_t^2q_k - (\overline\nu - \nu_{k+1})\partial_t^2\overline q
		&&\in H^1(I, L^2(\Omega))	
		\\\notag
		\rho^{VI} 
		&=\int_I\partial_t^2(p_k(t)-\overline p(t))q_{k+1}(t) + \partial_t^2(p_{k+1}(t) - p_k(t))q_k(t) + \partial_t^2(\overline p(t) - p_{k+1}(t))\overline q(t)\d t
		&&\in L^\infty(\Omega),
	\end{align}
	satisfying $\rho^{st}(0) = \rho^{adj}(T) = 0$. Using \cref{assumption4}, \cref{lemma: estimates} for $k$ and $k+1$, we obtain that 
	\begin{align*}
		&\|\rho^{st}\|_{H^1(I, L^2(\Omega))} + \|\rho^{adj}\|_{H^1(I, L^2(\Omega))} + \|\rho^{VI}\|_{L^\infty(\Omega)}
		\\
		&\leq \|\nu_k - \overline\nu\|_{L^2(\Omega)}\|\partial_t^2 p_{k+1}\|_{H^1(I, L^\infty(\Omega))} + \|\nu_{k+1} - \nu_k\|_{L^2(\Omega)}\|\partial_t^2p_k\|_{H^1(I, L^\infty(\Omega))} 
		\\
		&\quad + \|\overline\nu - \nu_{k+1}\|_{L^2(\Omega)}\|\partial_t^2\overline p\|_{H^1(I, L^\infty(\Omega))} + \|\nu_k - \overline\nu\|_{L^2(\Omega)}\|\partial_t^2 q_{k+1}\|_{H^1(I, L^\infty(\Omega))} 
		\\
		&\quad+ \|\nu_{k+1} - \nu_k\|_{L^2(\Omega)}\|\partial_t^2q_k\|_{H^1(I, L^\infty(\Omega))} + \|\overline\nu - \nu_{k+1}\|_{L^2(\Omega)}\|\partial_t^2\overline q\|_{H^1(I, L^\infty(\Omega))}
		\\
		&\quad +\|\partial_t^2(p_k-\overline p)\|_{L^2(I, L^\infty(\Omega))}\|q_{k+1}\|_{L^2(I, L^\infty(\Omega))} + \|\partial_t^2(p_{k+1} - p_k)\|_{L^2(I, L^\infty(\Omega))}\|q_k\|_{L^2(I, L^\infty(\Omega))} 
		\\
		&\quad+\|\partial_t^2(\overline p - p_{k+1})\|_{L^2(I, L^\infty(\Omega))}\|\overline q\|_{L^2(I, L^\infty(\Omega))}
		\\
		&\hspace{.1cm}\leq ((8C_0 + 4\overline C)(3k+5)! ^2 + (3c_1C_0 + c_1\overline C)(3k+5)!^4)(\|\overline\nu - \nu_{k+1}\|_{L^2(\Omega)} + \|\overline\nu - \nu_k\|_{L^2(\Omega)})
		\\
		&\hspace{-.3cm}\underbrace{\leq}_{\cref{assumptionk},\cref{claim nu k+1}} \frac{1}{4\delta}(8C_0 + 4\overline C + 3c_1C_0 + c_1\overline C)(3k+5)!^4(\underbrace{b_{k+1} + b_k}_{\leq 2b_{k-1}})
		\\
		&\hspace{.1cm}\leq\frac{1}{2\delta}(8C_0 + 4\overline C + 3c_1C_0 + c_1\overline C)(3k+5)!^4b_{k-1}\underbrace{\leq}_{\cref{ieq gamma bar}}\frac{1}{2\delta}(8C_0 + 4\overline C + 3c_1C_0 + c_1\overline C)\overline\gamma\underbrace{<}_{\cref{assumption gammas}} \frac{\tau}{c_L}.
	\end{align*}
	Therefore, by \cref{lemma: firstorderrho}, $(\nu_{k+1}, p_{k+1}, q_{k+1})$ solves \cref{system: linpertOC} with the perturbation terms \cref{perturbations iteration}. As a consequence, $(\nu_{k+1}, p_{k+1}, q_{k+1})$ satisfies \cref{iteration}. Assume that $(\tilde\nu_{k+1}, \tilde p_{k+1}, \tilde q_{k+1})\in\mathcal V_{ad}^\tau\times X_0\times X_T$ is another solution to \cref{iteration}. Then, $(\tilde\nu_{k+1}, \tilde p_{k+1}, \tilde q_{k+1})$ is also a fixed point of $S_k$, and consequently $\tilde\nu_{k+1}\in\mathcal V_{ad}^\tau$ is a fixed point of $(I_\nu\circ S_k\circ T_k)$. By the uniqueness of the fixed point of $(I_\nu\circ S_k\circ T_k)$ (see \cref{proposition contraction}), it follows that $\tilde\nu_{k+1} = \nu_{k+1}$. Furthermore, by the uniqueness of the solutions to the PDEs in \cref{iteration}, we obtain $\tilde p_{k+1} = p_{k+1}$ and $\tilde q_{k+1} = q_{k+1}$. Therefore, $(\nu_{k+1}, p_{k+1}, q_{k+1})$ is the unique solution to \cref{iteration} in $\mathcal V_{ad}^\tau\times X_0\times X_T$.
\end{proof}

We underline that \cref{proposition uniqueness} guarantees the existence of a unique solution of \cref{iteration} in $\mathcal V_{ad}^\tau \times X_0 \times X_T$. This result, however, does not imply uniqueness in $\mathcal V_{ad}\times X_0\times X_T$ since the first component of every solution to \cref{iteration} does not necessarily lie in $\mathcal V_{ad}^\tau$. In the following, we consider the subset 
\begin{equation}\label{defU}
	\overline{\mathcal U} \coloneqq \left\{(\nu,p,q) \in \mathcal V_{ad} \times X_0 \times X_T \quad | \quad \|\partial_t^2(\overline p -p)\|_{L^2(I,L^\infty(\Omega))} \leq \frac{c_1\overline \gamma}{\delta} , \quad \|\overline q- q\|_{L^2(I,L^\infty(\Omega))} \leq \frac{c_1\overline \gamma}{\delta} \right\}.
\end{equation}
We call $(\nu,p,q) \in \mathcal V_{ad}\times X_0\times X_T$ a solution to \cref{iteration} in $\overline{\mathcal U}$ if it solves \cref{iteration} and lies in $\overline{\mathcal U}$. As a consequence of \cref{proposition contraction,proposition uniqueness}, we show that \cref{iteration} turns to admit a unique solution in $\overline{\mathcal U}$.

\begin{corollary}\label{finalcor} 
	Under the assumptions of \cref{proposition uniqueness}, \cref{iteration} admits a unique solution in $\overline{\mathcal U}$ which coincides with the unique solution of \cref{iteration} in $\mathcal V_{ad}^\tau\times X_0\times X_T$. 
\end{corollary}

\begin{proof}
	Let $k \in \mathbb N_0$ and $(\nu_{k+1}, p_{k+1}, q_{k+1}) \in \mathcal V_{ad}^\tau \times X_0 \times X_T$ denote the unique solution to \cref{iteration} in $\mathcal V_{ad}^\tau \times X_0 \times X_T$ according to \cref{proposition uniqueness}. By \cref{assumption4} and \cref{ieq: quadratic1} (if $k=0$) as well as \cref{assumptionk} and \cref{ieq: quadratic2} (if $k \ge 1$), it holds that 
	\begin{equation}\label{proofuni1}
		\|\nu_{k}- \overline\nu\|_{L^2(\Omega)}\leq \frac{1}{4\delta}b_{k}, \quad \|\nu_{k+1}- \overline\nu\|_{L^2(\Omega)}\leq \frac{1}{4\delta}b_{k+1}
	\end{equation}
	with $b_k, b_{k+1}$ as in \cref{lemma sequences}. Furthermore, in view of \cref{proofuni1} in combination with \cref{lemma: estimates} (\cref{ieq 7}-\cref{ieq 8} with $k$ replaced by $k+1$) and since $\{b_n\}_{n\in\mathbb N_0}$ is decreasing, we have that 
	\begin{align*}
		&\|\partial_t^2(\overline p -p_{k+1})\|_{L^2(I,L^\infty(\Omega))}\underbrace{\leq}_{\cref{ieq 7}, \cref{proofuni1}} \frac{c_1}{2\delta}(3k+5)!^2 b_{k-1} \underbrace{\le}_{\cref{ieq gamma bar}} \frac{c_1\overline \gamma}{2\delta} \phantom{.} 
		\\
		&\| \overline q - q_{k+1}\|_{L^2(I,L^\infty(\Omega))} \qquad \underbrace{\leq}_{\cref{ieq 8}, \cref{proofuni1}} \frac{c_1}{2\delta}(3k+1)!^2 b_{k-1} \underbrace{\le}_{\cref{ieq gamma bar}} \frac{c_1\overline \gamma}{2\delta} . 
	\end{align*}
	Thus, we come to the conclusion that $(\nu_{k+1}, p_{k+1}, q_{k+1})\in\overline{\mathcal U}$, i.e., it is a solution to \cref{iteration} in $\overline{\mathcal U}$. Now, suppose that $(\nu,p,q)\in\overline{\mathcal U}$ is another solution to \cref{iteration}. It remains to prove that $\nu\in\mathcal V_{ad}^\tau$, from which it follows that $(\nu,p,q)=(\nu_{k+1}, p_{k+1}, q_{k+1})$, and so the uniqueness holds. To verify $\nu \in \mathcal V_{ad}^\tau$, we recall from the proof of \cref{lemma: firstorderrho} that 
	\begin{equation}\label{nuoptim}
		\overline\nu = \nu_- \text{ a.e. in } \mathscr A_\tau^+(\overline\nu) \quad \text{and} \quad \overline\nu = \nu_+ \text{ a.e. in }\mathscr A_\tau^-(\overline\nu)
	\end{equation}
	with $\mathscr A_\tau^+(\overline\nu) = \{x\in\Omega : -\int_{0}^{T}\partial_t^2\overline p(t,x)\overline q(t,x)\d t + \lambda\overline\nu(x) >\tau \}$ and $\mathscr A_\tau^-(\overline\nu) = \{x\in\Omega : -\int_{0}^{T}\partial_t^2\overline p(t,x)\overline q(t,x)\d t + \lambda\overline\nu(x) <-\tau \}$. Therefore, for a.e. $x\in\mathscr A_\tau^+(\overline\nu)$, we have that
	\begin{align}\label{proofuni2}
		\tau 
		&<-\int_{0}^{T}\partial_t^2\overline p(t,x) \overline q(t,x)\d t + \lambda \nu_-
		\\\notag
		&= -\int_0^T\partial_t^2 p_k(t,x) q(t,x) + \partial_t^2(p(t,x)- p_k(t,x))q_k(t,x)\d t + \lambda \nu(x) + \lambda (\nu_- - \nu(x))
		\\ \notag
		&\quad+ \int_0^T\!\partial^2_t p_k(t,x) (q(t,x) \!-\! \overline q (t,x)) \!+\! \partial_t^2(p(t,x) \!-\! p_k(t,x))(q_k(t,x) \!-\! \overline q (t,x)) \!+\! \partial_t^2 (p(t,x)\!-\!\overline p (t,x)) \overline q(t,x) \d t
		\\ \notag
		&\leq -\int_0^T\partial_t^2 p_k(t,x) q(t,x) + \partial_t^2(p(t,x)- p_k(t,x))q_k(t,x)\d t + \lambda \nu(x) 
		\\\notag 
		&\quad + ( \| \partial^2_t ( p_k -\overline p) \|_{L^2(I,L^\infty(\Omega))} + \| \partial^2_t \overline p \|_{L^2(I,L^\infty(\Omega))} )\|q- \overline q \|_{L^2(I,L^\infty(\Omega))} 
		\\\notag 
		&\quad + ( \| \partial^2_t ( p - \overline p) \|_{L^2(I,L^\infty(\Omega))} + \| \partial^2_t (\overline p - p_k) \|_{L^2(I,L^\infty(\Omega))} )\|q_k- \overline q \|_{L^2(I,L^\infty(\Omega))} \\\notag
		&\quad+ \| \partial^2_t (p -\overline p) \|_{L^2(I,L^\infty(\Omega))} \|\overline q \|_{L^2(I,L^\infty(\Omega))}.
	\end{align}
	If $k \ge 1$, we obtain from \cref{lemma: estimates} and the monotonicity of $\{b_n\}_{n\in\mathbb N_0}$ that
	\begin{align}\label{proofuni3}
		\|\partial_t^2( p_k -\overline p)\|_{L^2(I,L^\infty(\Omega))}
		&\underbrace{\leq}_{\cref{ieq 7}, \cref{assumptionk}} \frac{c_1}{2\delta}(3k+2)!^2 b_{k-1}\underbrace{\leq}_{\cref{ieq gamma bar}} \frac{c_1 \overline \gamma}{2 \delta} 
		\\\notag
		\| q_k -\overline q\|_{L^2(I,L^\infty(\Omega))}
		&\underbrace{\leq}_{\cref{ieq 8}, \cref{assumptionk}} \frac{c_1}{2\delta}(3k)!^2 b_{k-1} \leq \frac{c_1}{4\delta}(3k+5)!^4 b_{k-1} \underbrace{\le}_{\cref{ieq gamma bar}}\frac{c_1 \overline \gamma}{4 \delta}
	\end{align}	
	If $k=0$, \cref{assumption4} yields 
	\begin{equation}\label{proofuni4}
		\|\partial_t^2( p_0 -\overline p)\|_{L^2(I,L^\infty(\Omega))} \leq 4 \varepsilon \underbrace{\le}_{\cref{assumption epsilon}} \frac{c_1 \overline \gamma}{2 \delta} \quad \text{and} \quad \| q_0 - \overline q\|_{L^2(I,L^\infty(\Omega))} \leq \varepsilon \leq \frac{c_1 \overline \gamma}{8 \delta}
	\end{equation}	
	Applying \cref{defU}, \cref{proofuni3}, and \cref{proofuni4} to \cref{proofuni2} implies for a.e. $x\in\mathscr A_\tau^+(\overline\nu)$ that 
	\begin{align*} 
		\tau 
		&< -\int_0^T\partial_t^2 p_k(t,x) q(t,x) + \partial_t^2(p(t,x)- p_k(t,x))q_k(t,x)\d t + \lambda \nu(x) 
		\\
		&\qquad \qquad + ( \frac{c_1 \overline \gamma}{2 \delta} + \| \partial^2_t \overline p \|_{L^2(I,L^\infty(\Omega))}) \frac{c_1 \overline \gamma}{ \delta}+ ( \frac{c_1 \overline \gamma}{ \delta} + \frac{c_1 \overline \gamma}{2\delta} ) \frac{c_1 \overline \gamma}{4\delta} + \frac{c_1 \overline \gamma}{ \delta} \|\overline q \|_{L^2(I,L^\infty(\Omega))} 
		\\
		&\underbrace{\leq}_{\cref{assumption4}, \, \, \overline \gamma \leq 1} -\int_0^T\partial_t^2 p_k(t,x) q(t,x) + \partial_t^2(p(t,x)- p_k(t,x))q_k(t,x)\d t + \lambda \nu(x) + \overline \gamma ( 5 \frac{c_1 }{ \delta} \overline C + \frac{c_1^2 }{\delta^2}).
	\end{align*}	
	Since according to \cref{assumption gammas}, $\overline \gamma \leq \tau / ( 5 \frac{c_1 }{ \delta} \overline C + \frac{c_1^2 }{\delta^2} )$, it follows that 
	\begin{equation} \label{proofuni5}
		-\int_0^T\partial_t^2 p_k(t,x) q(t,x) + \partial_t^2(p(t,x)- p_k(t,x))q_k(t,x)\d t + \lambda \nu(x) >0 \quad \text{ for a.e. } x\in\mathscr A_\tau^+(\overline\nu).
	\end{equation}
	In a completely analogous way, we deduce from the definition of $\mathscr A_\tau^-(\overline\nu) $ and \cref{nuoptim} that 
	\begin{equation}\label{proofuni6}
		-\int_0^T\partial_t^2 p_k(t,x) q(t,x) + \partial_t^2(p(t,x)- p_k(t,x))q_k(t,x)\d t + \lambda \nu(x) <0 \quad \text{ for a.e. } x\in\mathscr A_\tau^-(\overline\nu).
	\end{equation}
	Now, since $\nu$ satisfies the variational inequality in \cref{iteration}, a well-known result (see \cite[Lemma 2.26]{troeltzsch10}) imply due to \cref{proofuni5} and \cref{proofuni6} that 
	\begin{equation*}
		\nu = \nu_- \text{ a.e. in } \mathscr A_\tau^+(\overline\nu) \quad \text{and} \quad \nu = \nu_+ \text{ a.e. in }\mathscr A_\tau^-(\overline\nu) \quad \underbrace{\Rightarrow}_{\cref{nuoptim}} \quad \nu \in \mathcal V_{ad}^\tau.
	\end{equation*}
	In conclusion, the assertion is valid. 
\end{proof}

Differently from the parabolic case, we cannot prove the quadratic (Q-)convergence of \cref{alg:sqp}. As a remedy, the proposed two-step estimation process \cref{ieq: quadratic2} eventually enables us to prove R-superlinear convergence, i.e., the error is dominated by some scalar-valued sequence converging superlinearly to zero \cite[page 620]{nocedal99}. This final result is proven in the following by making use of the previous propositions and corollary:

\begin{theorem}
	\label{theorem: convergence}
	Let \cref{assumption4} be satisfied. Then, \cref{alg:sqp} generates a sequence $\{(\nu_{k+1},p_{k+1}, q_{k+1})\}_{k \in \mathbb N_0}$ of unique solutions to \cref{iteration} in $\overline{\mathcal U}$ satisfying
	\begin{equation*}
		\|\nu_{k+1} - \overline\nu\|_{L^2(\Omega)}\leq \delta(3k+5)!^4(\|\overline\nu - \nu_k\|_{L^2(\Omega)} + \|\overline\nu - \nu_{k-1}\|_{L^2(\Omega)})^2 \quad \forall k\in\mathbb N.
	\end{equation*}
	Furthermore, $\{\nu_k\}_{k\in\mathbb N_0}$ converges R-superlinearly towards the solution $\overline\nu$ to \cref{P} with
	\begin{equation}\label{convergence2}
		\|\nu_k- \overline\nu\|_{L^2(\Omega)}\leq\frac{1}{4\delta}\overline\gamma^{\sqrt{2}^k}\quad\forall k\in\mathbb N.
	\end{equation}
\end{theorem}

\begin{proof}
	Due to \cref{proposition contraction,proposition uniqueness} and \cref{finalcor}, \hyperref[iteration]{\textnormal{($\mathbb P_0$)}} admits a unique solution $(\nu_1, p_1, q_1)$ in $\overline {\mathcal U}$ satisfying
	\begin{equation}\label{ieq nu1}
		\|\overline\nu - \nu_1\|_{L^2(\Omega)}\leq \frac{1}{4\delta}b_1\underbrace{\leq}_{\cref{bk overline gamma}}\frac{1}{4\delta}\overline\gamma^{\sqrt{2}} 
	\end{equation}
	 with $b_k$ as in \cref{lemma sequences}. In view of \cref{ieq nu1} and \cref{assumption4}, $\nu_1$ and $(\nu_0, p_0, q_0)$ satisfy \cref{assumptionk} for $k=1$ such that \cref{proposition contraction,proposition uniqueness} and \cref{finalcor} imply 
	 that \cref{iteration} for $k=1$ admits a unique solution $(\nu_2, p_2,q_2)$ in $\overline{\mathcal U}$ satisfying
	\begin{equation*}
		\|\overline\nu - \nu_2\|_{L^2(\Omega)}\leq \delta 8!^4(\|\overline\nu - \nu_1\|_{L^2(\Omega)} + \|\overline\nu - \nu_0\|_{L^2(\Omega)})^2,\qquad \|\overline\nu - \nu_2\|_{L^2(\Omega)}\leq\frac{1}{4\delta}b_2 \underbrace{\leq}_{\cref{bk overline gamma}} \frac{1}{4\delta}\overline\gamma^{\sqrt{2}^{2}}.	\end{equation*} 
	Moreover, since $\nu_1$ and $(\nu_0, p_0, q_0)$ satisfy \cref{assumptionk} for $k=1$, \cref{lemma: estimates} implies $\|\partial_t^lp_1\|_{L^2(I,L^\infty(\Omega))} \leq C_0(l+3)!^2$ and $\|\partial_t^lq_1\|_{L^2(I,L^\infty(\Omega))} \leq C_0(l+3)!^2$ for all $l\in\mathbb N_0$. Thus, $\nu_2$ and $(\nu_1,p_1,q_1)$ satisfies \cref{assumptionk} for $k=2$. Now, suppose that $k\geq 2$, $(\nu_{k},p_k,q_k)$ is the unique solution of $(\mathbb P_{k-1})$, and $\nu_k$ and $(\nu_{k-1}, p_{k-1}, q_{k-1})$ satisfies \cref{assumptionk}. Then, \cref{proposition contraction}, \cref{proposition uniqueness}, and \cref{finalcor} imply that \cref{iteration} admits a unique solution $(\nu_{k+1}, p_{k+1},q_{k+1})$ in $\overline{\mathcal U}$ satisfying
	\begin{equation*}
		\|\overline\nu - \nu_{k+1}\|_{L^2(\Omega)}\leq\delta (3k+5)!^4(\|\nu_k - \overline\nu\|_{L^2(\Omega)} + \|\nu_{k-1} - \overline\nu\|_{L^2(\Omega)})^2,\quad\|\overline\nu - \nu_{k+1}\|_{L^2(\Omega)}\leq\frac{1}{4\delta}b_{k+1}\underbrace{\leq}_{\cref{bk overline gamma}} \frac{1}{4\delta}\overline\gamma^{\sqrt{2}^{k+1}}.
	\end{equation*}
	Again, since $\nu_k$ and $(\nu_{k-1}, p_{k-1}, q_{k-1})$ satisfies \cref{assumptionk}, \cref{lemma: estimates} implies that $\|\partial_t^lp_k\|_{L^2(I,L^\infty(\Omega))} \leq C_0(l+3k)!^2$ and $\|\partial_t^lq_k\|_{L^2(I,L^\infty(\Omega))} \leq C_0(l+3k)!^2$ for all $l\in\mathbb N_0$. Thus, $\nu_{k+1}$ and $(\nu_{k},p_k,q_k)$ satisfies $(\textnormal{A}_{k+1})$. In conclusion, the claim follows by induction.
\end{proof}

\subsection*{Acknowledgments} 
The authors are very thankful to Prof. Daniel Wachsmuth (University of W\"urzburg) for the insightful discussion and his valuable suggestion regarding a relaxion of the previous version of \cref{assumption4}. Furthermore, special thanks go to Prof. Mariano Mateos (University of Oviedo) for the insightful discussion regarding the uniqueness in the SQP algorithm that motivates us to establish \cref{finalcor}.

\appendix
\section{Appendix}

\subsection{Proof of \texorpdfstring{\cref{Lipschitzproperties}}{} in \texorpdfstring{\cref{theorem:lipschitzproperty}}{} }	
\label{subsection proof pertubation}

From \cref{ieq:nurho00}, we know that
\begin{align}
	\label{ieq:nurho}
		&\alpha\|\nu_\rho - \nu_{\widetilde\rho}\|_{L^2(\Omega)}^2\leq D_{(\nu,p)}^2\mathcal L(\overline \nu, \overline p, \overline q)(\nu_\rho - \nu_{\widetilde\rho},p_\rho - p_{\widetilde\rho} - \hat p_{\rho,\widetilde\rho})^2
		\\\notag
		&\underbrace{=}_{\cref{def: lagrangian}}\sum_{i=1}^m(a_i (p_\rho - p_{\widetilde\rho}),p_\rho - p_{\widetilde\rho})_{L^2(I,L^2(\Omega))} + \lambda\|\nu_\rho - \nu_{\widetilde\rho}\|_{L^2(\Omega)}^2 - 2((\nu_\rho - \nu_{\widetilde\rho})\partial_t^2(p_\rho - p_{\widetilde\rho}),\overline q)_{L^2(I,L^2(\Omega))}
		\\\notag
		&\hspace{.8cm} + \sum_{i=1}^m(a_i \hat p_{\rho,\widetilde\rho},\hat p_{\rho,\widetilde\rho})_{L^2(I,L^2(\Omega))} - 2\sum_{i=1}^m(a_i \hat p_{\rho,\widetilde\rho},(p_\rho - p_{\widetilde\rho}))_{L^2(I,L^2(\Omega))} + 2((\nu_\rho - \nu_{\widetilde\rho})\hat p_{\rho,\widetilde\rho},\partial_t^2\overline q)_{L^2(I,L^2(\Omega))}.
\end{align} 
For the first term on the right-hand side of \cref{ieq:nurho}, it holds that
\begin{equation*}
	\begin{aligned}
		&\sum_{i=1}^m(a_i(p_\rho - p_{\widetilde\rho}),p_\rho - p_{\widetilde\rho})_{L^2(I,L^2(\Omega))} 
		\\
		&\underbrace{=}_{\cref{system: adjointdifference}}(\overline\nu\partial_t^2(q_\rho- q_{\widetilde\rho}) - \Delta(q_\rho- q_{\widetilde\rho}) - \eta\partial_t(q_\rho- q_{\widetilde\rho}) + (\nu_\rho- \nu_{\widetilde\rho})\partial_t^2\overline q - \rho^{adj} + \widetilde\rho^{adj},p_\rho - p_{\widetilde\rho})_{L^2(I,L^2(\Omega))}
		\\
		&\hspace{1.4mm}=(q_\rho - q_{\widetilde\rho},\overline\nu\partial_t^2(p_\rho - p_{\widetilde\rho}) - \Delta(p_\rho - p_{\widetilde\rho}) + \eta\partial_t(p_\rho - p_{\widetilde\rho}))_{L^2(I, L^2(\Omega))} + (\overline q, (\nu_\rho- \nu_{\widetilde\rho})\partial_t^2(p_\rho - p_{\widetilde\rho}))_{L^2(I,L^2(\Omega))}
		\\
		&\qquad - (\rho^{adj} - \widetilde\rho^{adj},p_\rho - p_{\widetilde\rho})_{L^2(I,L^2(\Omega))}
		\\
		&\underbrace{=}_{\cref{system: statedifference}}(q_\rho - q_{\widetilde\rho},-(\nu_\rho - \nu_{\widetilde\rho})\partial_t^2\overline p + \rho^{st} - \widetilde\rho^{st}) + (\overline q,(\nu_\rho- \nu_{\widetilde\rho})\partial_t^2(p_\rho - p_{\widetilde\rho}))_{L^2(I,L^2(\Omega))} - (\rho^{adj} - \widetilde\rho^{adj},p_\rho - p_{\widetilde\rho})_{L^2(I,L^2(\Omega))}.
	\end{aligned}
\end{equation*}
Applying this identity to \cref{ieq:nurho}, we obtain that
\begin{align}
	\label{ieq: ssc}
		&\alpha\|\nu_\rho - \nu_{\widetilde\rho}\|_{L^2(\Omega)}^2
		\\\notag
		&\leq -((\nu_\rho - \nu_{\widetilde\rho})\partial_t^2\overline p,q_\rho - q_{\widetilde\rho})_{L^2(I,L^2(\Omega))} + (q_\rho - q_{\widetilde\rho},\rho^{st} - \widetilde\rho^{st})_{L^2(I,L^2(\Omega))} - (\rho^{adj} - \widetilde\rho^{adj},p_\rho - p_{\widetilde\rho})_{L^2(I,L^2(\Omega))} 
		\\\notag
		&\quad + \lambda\|\nu_\rho - \nu_{\widetilde\rho}\|_{L^2(\Omega)}^2 - ((\nu_\rho- \nu_{\widetilde\rho})\partial_t^2(p_\rho - p_{\widetilde\rho}),\overline q)_{L^2(I,L^2(\Omega))}+ \sum_{i=1}^m(a_i \hat p_{\rho,\widetilde\rho},\hat p_{\rho,\widetilde\rho})_{L^2(I,L^2(\Omega))}
		\\\notag
		&\quad - 2\sum_{i=1}^m(a_i \hat p_{\rho,\widetilde\rho},(p_\rho - p_{\widetilde\rho}))_{L^2(I,L^2(\Omega))} + 2((\nu_\rho - \nu_{\widetilde\rho})\hat p_{\rho,\widetilde\rho},\partial_t^2\overline q)_{L^2(I,L^2(\Omega))}.
\end{align} 
Testing the variational inequality in \cref{system: linpertOC2} for $(\nu_\rho,p_\rho,q_\rho)$ (resp. $(\nu_{\widetilde\rho}, p_{\widetilde\rho}, q_{\widetilde\rho})$) with $\tilde\nu = \nu_{\widetilde\rho}$ (resp. $\tilde\nu = \nu_\rho$) and adding the resulting two inequalites leads to
\begin{equation*}
	\left(-\int_0^T\partial_t^2\overline p(t)(q_\rho(t) - q_{\widetilde\rho}(t)) + \partial_t^2(p_\rho(t) - p_{\widetilde\rho}(t))\overline q(t)\d t + \lambda(\nu_\rho -\nu_{\widetilde\rho}), \nu_{\widetilde\rho} - \nu_\rho\right)_{L^2(\Omega)}\geq (\rho^{VI}-\widetilde\rho^{VI},\nu_{\widetilde\rho}-\nu_{\rho})_{L^2(\Omega)}.
\end{equation*}
Rearranging yields that
\begin{equation}\label{ieq: vi}
	\lambda\|\nu_\rho - \nu_{\widetilde\rho}\|_{L^2(\Omega)}^2 - ((\nu_\rho-\nu_{\widetilde\rho})\partial_t^2\overline p ,q_\rho- q_{\widetilde\rho})_{L^2(I,L^2(\Omega))} - ((\nu_\rho -\nu_{\widetilde\rho})\partial_t^2(p_\rho -p_{\widetilde\rho}),\overline q)_{L^2(I,L^2(\Omega))} \leq (\rho^{VI}-\widetilde\rho^{VI},\nu_\rho-\nu_{\widetilde\rho})_{L^2(\Omega)}.
\end{equation}
Combining \cref{ieq: ssc} and \cref{ieq: vi}, we obtain that
\begin{align}
	\label{ieq: alpha1}
	&\alpha\|\nu_\rho - \nu_{\widetilde\rho}\|_{L^2(\Omega)}^2
	\\\notag
	&\leq (q_\rho - q_{\widetilde\rho},\rho^{st} - \widetilde\rho^{st})_{L^2(I,L^2(\Omega))} - (\rho^{adj} - \widetilde\rho^{adj},p_\rho - p_{\widetilde\rho})_{L^2(I,L^2(\Omega))} + (\rho^{VI}-\widetilde\rho^{VI},\nu_\rho-\nu_{\widetilde\rho})_{L^2(\Omega)}
	\\\notag
	&\quad+ \sum_{i=1}^m(a_i \hat p_{\rho,\widetilde\rho},\hat p_{\rho,\widetilde\rho})_{L^2(I,L^2(\Omega))} - 2\sum_{i=1}^m(a_i \hat p_{\rho,\widetilde\rho},(p_\rho - p_{\widetilde\rho}))_{L^2(I,L^2(\Omega))}+ 2((\nu_\rho - \nu_{\widetilde\rho})\partial_t^2\hat p_{\rho,\widetilde\rho},\overline q)_{L^2(I,L^2(\Omega))}
	\\\notag
	&\leq \|q_\rho - q_{\widetilde\rho}\|_{L^2(I,L^2(\Omega))}\|\rho^{st} - \widetilde\rho^{st}\|_{L^2(I,L^2(\Omega))} + \|\rho^{adj} - \widetilde\rho^{adj}\|_{L^2(I,L^2(\Omega))}\|p_\rho - p_{\widetilde\rho}\|_{L^2(I,L^2(\Omega))} 
	\\\notag
	&\quad + \|\rho^{VI}-\widetilde\rho^{VI}\|_{L^2(\Omega)}\|\nu_\rho-\nu_{\widetilde\rho}\|_{L^2(\Omega)} + \sum_{i=1}^{m}\|a_i\|_{L^\infty(I, L^\infty(\Omega))}\|\hat p_{\rho,\widetilde\rho}\|_{L^2(I, L^2(\Omega))}^2 
	\\\notag
	&\quad+ 2\sum_{i=1}^{m}\|a_i\|_{L^\infty(I, L^\infty(\Omega))}\|\hat p_{\rho,\widetilde\rho}\|_{L^2(I, L^2(\Omega))}\|p_\rho - p_{\widetilde\rho}\|_{L^2(I,L^2(\Omega))} 
	\\\notag
	&\quad+ 2\|\nu_\rho - \nu_{\widetilde\rho}\|_{L^2(\Omega)}\|\hat p_{\rho,\widetilde\rho}\|_{L^2(I, L^2(\Omega))}\|\partial_t^2\overline q\|_{L^2(I,L^\infty(\Omega))}.
\end{align}
Applying \cref{lemma: abstract} to \cref{prhorho} yields for $G(t)\coloneqq\int_0^{t}\rho^{st}(s) - \widetilde\rho^{st}(s)ds$ that 
\begin{equation}\label{ieq: alpha2}
	\|\hat p_{\rho,\widetilde\rho}\|_{L^2(I, L^2(\Omega))}\leq \sqrt{T}c\|G\|_{L^1(I, L^2(\Omega))}\leq T^{\frac{3}{2}}c\|\rho^{st} - \widetilde\rho^{st}\|_{L^1(I, L^2(\Omega))} \leq T^2c\|\rho^{st} - \widetilde\rho^{st}\|_{L^2(I, L^2(\Omega))}
\end{equation}
with $c\coloneqq \nu_{\min}^{-1}\frac{\max\left\{\sqrt{\nu_{\max}},1\right\}}{\min\{\sqrt{\nu_{\min}},1\}}$. Analogously, applying \cref{lemma: abstract} to \cref{system: statedifference} and \cref{system: adjointdifference}, we obtain that
\begin{align}\label{ieq: alpha3}
	&\|p_\rho - p_{\widetilde\rho}\|_{L^2(I,L^2(\Omega))}\leq T^2c(\|\nu_\rho- \nu_{\widetilde\rho}\|_{L^2(\Omega)}\|\partial_t^2\overline p\|_{L^2(I,L^\infty(\Omega))} + \|\rho^{st}-\widetilde\rho^{st}\|_{L^2(I,L^2(\Omega))})
	\\\label{ieq: alpha4}
	&\|q_\rho - q_{\widetilde\rho}\|_{L^2(I,L^2(\Omega))}
	\\\notag
	&\hspace{1.4mm}\leq T^2c\Bigg(\sum_{i=1}^m\|a_i\|_{L^\infty(I, L^\infty(\Omega))}\|p_\rho - p_{\widetilde\rho}\|_{L^2(I,L^2(\Omega))} +\|\nu_\rho- \nu_{\widetilde\rho}\|_{L^2(\Omega)}\|\partial_t^2\overline q\|_{L^2(I, L^\infty(\Omega))} + \|\rho^{adj}-\widetilde\rho^{adj}\|_{L^2(I,L^2(\Omega))}\Bigg)
	\\\notag
	&\underbrace{\leq}_{\cref{ieq: alpha3}} T^2c\Bigg(\sum_{i=1}^m\|a_i\|_{L^\infty(I, L^\infty(\Omega))}T^2c(\|\nu_\rho- \nu_{\widetilde\rho}\|_{L^2(\Omega)}\|\partial_t^2\overline p\|_{L^2(I,L^\infty(\Omega))} + \|\rho^{st}-\widetilde\rho^{st}\|_{L^2(I,L^2(\Omega))}) 
	\\
	&\notag\hspace{1.7cm} +\|\nu_\rho- \nu_{\widetilde\rho}\|_{L^2(\Omega)}\|\partial_t^2\overline q\|_{L^2(I,L^\infty(\Omega))} + \|\rho^{adj}-\widetilde\rho^{adj}\|_{L^2(I,L^2(\Omega))}\Bigg).
\end{align}
Therefore, applying \cref{ieq: alpha2}-\cref{ieq: alpha4} to \cref{ieq: alpha1} provides that
\begin{align*}
	&\alpha\|\nu_\rho - \nu_{\widetilde\rho}\|_{L^2(\Omega)}^2 
	\\
	&\leq T^2c\Bigg(\sum_{i=1}^m\|a_i\|_{L^\infty(I, L^\infty(\Omega))}T^2c\left(\|\nu_\rho- \nu_{\widetilde\rho}\|_{L^2(\Omega)}\|\partial_t^2\overline p\|_{L^2(I,L^\infty(\Omega))} + \|\rho^{st}-\widetilde\rho^{st}\|_{L^2(I,L^2(\Omega))}\right) 
	\\
	&\hspace{1.4cm} +\|\nu_\rho- \nu_{\widetilde\rho}\|_{L^2(\Omega)}\|\partial_t^2\overline q\|_{L^2(I,L^\infty(\Omega))} + \|\rho^{adj}-\widetilde\rho^{adj}\|_{L^2(I,L^2(\Omega))}\Bigg)\|\rho^{st} - \widetilde\rho^{st}\|_{L^2(I,L^2(\Omega))} 
	\\
	&\quad + \|\rho^{adj} - \widetilde\rho^{adj}\|_{L^2(I,L^2(\Omega))}T^2c(\|\nu_\rho- \nu_{\widetilde\rho}\|_{L^2(\Omega)}\|\partial_t^2\overline p\|_{L^2(I,L^\infty(\Omega))} + \|\rho^{st}-\widetilde\rho^{st}\|_{L^2(I,L^2(\Omega))})
	\\
	&\quad + \|\rho^{VI}-\widetilde\rho^{VI}\|_{L^2(\Omega)}\|\nu_\rho-\nu_{\widetilde\rho}\|_{L^2(\Omega)} + T^4c^2\sum_{i=1}^{m}\|a_i\|_{L^\infty(I, L^\infty(\Omega))}\|\rho^{st} - \widetilde\rho^{st}\|_{L^2(I, L^2(\Omega))}^2 
	\\
	&\quad+ 2\sum_{i=1}^{m}\|a_i\|_{L^\infty(I, L^\infty(\Omega))}T^4c^2\|\rho^{st} - \widetilde\rho^{st}\|_{L^2(I, L^2(\Omega))}(\|\nu_\rho- \nu_{\widetilde\rho}\|_{L^2(\Omega)}\|\partial_t^2\overline p\|_{L^2(I,L^\infty(\Omega))} + \|\rho^{st}-\widetilde\rho^{st}\|_{L^2(I,L^2(\Omega))})
	\\
	&\quad+ 2\|\nu_\rho - \nu_{\widetilde\rho}\|_{L^2(\Omega)}T^2c\|\rho^{st} - \widetilde\rho^{st}\|_{L^2(I, L^2(\Omega))}\|\partial_t^2\overline q\|_{L^2(I,L^\infty(\Omega))}
	\\
	&\leq C_1\|\nu_\rho - \nu_{\widetilde\rho}\|_{L^2(\Omega)}\|\rho^{st} - \widetilde\rho^{st}\|_{L^2(I,L^2(\Omega))} + C_2\|\rho^{st} - \widetilde\rho^{st}\|_{L^2(I,L^2(\Omega))}^2
	\\
	&\quad + C_3\|\rho^{adj} - \widetilde\rho^{adj}\|_{L^2(I,L^2(\Omega))}\|\rho^{st} - \widetilde\rho^{st}\|_{L^2(I,L^2(\Omega))} + C_4\|\nu_\rho - \nu_{\widetilde\rho}\|_{L^2(\Omega)}\|\rho^{adj} - \widetilde\rho^{adj}\|_{L^2(I,L^2(\Omega))} 
	\\
	&\quad + \|\rho^{VI}-\widetilde\rho^{VI}\|_{L^2(\Omega)}\|\nu_\rho-\nu_{\widetilde\rho}\|_{L^2(\Omega)}
\end{align*}
with the constants
\begin{align*}
	&C_1\coloneqq 3T^4c^2\sum_{i=1}^m\|a_i\|_{L^\infty(I, L^\infty(\Omega))}\|\partial_t^2\overline p\|_{L^2(I, L^\infty(\Omega))} +3T^2c\|\partial_t^2\overline q\|_{L^2(I, L^\infty(\Omega))},	
	&&C_2 \coloneqq 4T^4c^2\sum_{i=1}^m\|a_i\|_{L^\infty(I, L^\infty(\Omega))},
	\\
	&C_3\coloneqq 2T^2c,
	&&C_4\coloneqq T^2c\|\partial_t^2\overline p\|_{L^2(I,L^\infty(\Omega))}.
\end{align*}
Using Young's inequality, we obtain that
\begin{align*}
	&\alpha\|\nu_\rho - \nu_{\widetilde\rho}\|_{L^2(\Omega)}^2
	\\
	&\leq \frac{\alpha}{4}\|\nu_\rho - \nu_{\widetilde\rho}\|_{L^2(\Omega)}^2 + \left(\frac{1}{\alpha}C_1^2 \!+\! C_2\right)\|\rho^{st} - \widetilde\rho^{st}\|_{L^2(I,L^2(\Omega))}^2 +\frac{C_3}{2}\|\rho^{adj} - \widetilde\rho^{adj}\|_{L^2(I,L^2(\Omega))}^2+ \frac{C_3}{2}\|\rho^{st} - \widetilde\rho^{st}\|_{L^2(I,L^2(\Omega))}^2 
	\\
	&\quad + \frac{\alpha}{4}\|\nu_\rho - \nu_{\widetilde\rho}\|_{L^2(\Omega)}^2 + \frac{1}{\alpha}C_4^2\|\rho^{adj} - \widetilde\rho^{adj}\|_{L^2(I,L^2(\Omega))}^2 + \frac{1}{\alpha}\|\rho^{VI}-\widetilde\rho^{VI}\|_{L^2(\Omega)}^2 + \frac{\alpha}{4}\|\nu_\rho-\nu_{\widetilde\rho}\|_{L^2(\Omega)}^2,
\end{align*}
leading to 
\begin{align}\label{ieq: controlfinal}
	&\frac{\alpha}{4}\|\nu_\rho - \nu_{\widetilde\rho}\|_{L^2(\Omega)}^2
	\\\notag
	&\leq\left(\frac{1}{\alpha}C_1^2 + C_2 + \frac{C_3}{2}\right)\|\rho^{st} - \widetilde\rho^{st}\|_{L^2(I,L^2(\Omega))}^2 +\left(\frac{C_3}{2} + \frac{1}{\alpha}C_4^2\right)\|\rho^{adj} - \widetilde\rho^{adj}\|_{L^2(I,L^2(\Omega))}^2 + \frac{1}{\alpha}\|\rho^{VI}-\widetilde\rho^{VI}\|_{L^2(\Omega)}^2.
\end{align}
Therefore, \cref{Lipschitzproperties} is valid.

\subsection{Proof of \texorpdfstring{\cref{lemma: estimates}}{}}
\label{subsection proof lemma estimates}

\begin{proof}
	Let $\nu,\tilde\nu\in\mathcal V_{ad}$, $(\overline\nu,\check p, \check q)\coloneqq T_0(\overline\nu)$, $(\nu,p, q)\coloneqq T_0(\nu)$, and $(\tilde\nu,\tilde p, \tilde q)\coloneqq T_0(\tilde\nu)$. Furthermore, let $l\in\mathbb N_0$. We first note from \cref{system: statetk} that $\check p - \overline p$ solves 
	\begin{equation*}
		\left\{\begin{aligned}
			&\nu_0\partial_t^2 (\check p - \overline p) - \Delta(\check p - \overline p) + \eta\partial_t(\check p - \overline p) = -(\overline\nu - \nu_0)\partial_t^2(p_0 - \overline p)
			&&\text{in } I\times \Omega
			\\
			&\partial_n(\check p - \overline p) = 0 
			&&\text{on }I\times\Gamma_N
			\\
			&\check p - \overline p = 0 
			&&\text{on }I\times\Gamma_D	
			\\
			&(\check p - \overline p, \partial_t(\check p - \overline p))(0) = (0, 0) &&\text{in } \Omega,
		\end{aligned}\right.
	\end{equation*}
	such that \cref{lemma: abstract}, \cref{lemma: stampaccia}, and \cref{assumption4} yield that
	\begin{align}\notag
		\|\partial_t^l (\check p - \overline p)\|_{L^2(I,L^2(\Omega))}
		&\!\leq\! c\sqrt{T}\|\overline\nu - \nu_0\|_{L^2(\Omega)}\|\partial_t^{l+1}(p_0 - \overline p)\|_{L^1(I, L^\infty(\Omega))}\!\leq\! cT\|\overline\nu - \nu_0\|_{L^2(\Omega)}\|\partial_t^{l+1}(p_0 - \overline p)\|_{L^2(I, L^\infty(\Omega))}
		\\\label{ieq p checkb}
		&\leq cTC_0(l+1)! ^2\varepsilon^2\underbrace{\leq}_{\cref{assumption epsilon}} c TC_0\frac{\gamma}{4\delta} (l+1)! ^2\varepsilon
	\end{align}
	and
	\begin{align}\notag
		\|\partial_t^l (\check p - \overline p)\|_{L^2(I,L^\infty(\Omega))}&\leq \hat c(\|\overline\nu - \nu_0\|_{L^2(\Omega)}\|\partial_t^{l+2} (p_0 - \overline p)\|_{L^2(I,L^\infty(\Omega))} + \|\overline\nu - \nu_0\|_{L^2(\Omega)}\|\partial_t^{l+3} (p_0 - \overline p)\|_{L^2(I,L^\infty(\Omega))})
		\\\label{ieq p check}
		&\leq 2\hat c\frac{\gamma}{4\delta} C_0(l+3)! ^2\varepsilon.
	\end{align}
	Consequently, by the triangular inequality and \cref{assumption4} along with \cref{{ieq p check}}, we obtain that 
	\begin{equation*}
		\|\partial_t^l (\check p - p_0)\|_{L^2(I,L^\infty(\Omega))}\leq \|\partial_t^l (\check p - \overline p)\|_{L^2(I,L^\infty(\Omega))} + \|\partial_t^l (\overline p - p_0)\|_{L^2(I,L^\infty(\Omega))}\leq (\hat c\frac{\gamma}{2\delta} + 1)C_0(l+3)! ^2\varepsilon.
	\end{equation*}
	From the definition of $c_0$ (see \cref{definition delta}), this implies \cref{ieq 1}. Similarly, due to \cref{first order optimality} and \cref{system: adjointtk}, $\check q - \overline q$ satisfies 
	\begin{equation*}
		\left\{\begin{aligned}
			&\nu_0\partial_t^2 (\check q - \overline q) - \Delta(\check q - \overline q) - \eta\partial_t(\check q - \overline q) = \sum_{i=1}^ma_i(\check p - \overline p) -(\overline\nu - \nu_0)\partial_t^2(q_0 - \overline q)
			&&\text{in } I\times \Omega
			\\
			&\partial_n(\check q - \overline q) = 0 
			&&\text{on }I\times\Gamma_N
			\\
			&\check q - \overline q = 0 
			&&\text{on }I\times\Gamma_D	
			\\
			&(\check q - \overline q, \partial_t(\check q - \overline q))(T) = (0, 0) 
			&&\text{in } \Omega,
		\end{aligned}\right.
	\end{equation*}
	such that \cref{lemma: stampaccia}, \cref{assumption4}, and \cref{ieq p checkb} yield that
	\begin{align}\label{ieq q check q bar2}
		&\|\partial_t^l (\check q - \overline q)\|_{L^2(I,L^\infty(\Omega))}
		\\\notag
		&\leq \hat c(\sum_{j=0}^l\binom{l}{j}\sum_{i=1}^m\|\partial_t^ja_i\|_{L^\infty(I, L^\infty(\Omega))}\|\partial_t^{l-j}(\check p - \overline p)\|_{L^2(I,L^2(\Omega))}+ \|\overline\nu - \nu_0\|_{L^2(\Omega)}\|\partial_t^{l+2} (q_0 - \overline q)\|_{L^2(I,L^\infty(\Omega))} 
		\\\notag
		&\qquad + \sum_{j=0}^{l+1}\binom{l+1}{j}\sum_{i=1}^m\|\partial_t^ja_i\|_{L^\infty(I,L^\infty(\Omega))}\|\partial_t^{l+1-j}(\check p - \overline p)\|_{L^2(I,L^2(\Omega))} + \|\overline\nu - \nu_0\|_{L^2(\Omega)}\|\partial_t^{l+3} (q_0 - \overline q)\|_{L^2(I,L^\infty(\Omega))})
		\\\notag
		&\leq \hat c(C_acTC_0\frac{\gamma}{4\delta}\sum_{j=0}^l\binom{l}{j}j! ^2(l-j+1)! ^2\varepsilon + \frac{\gamma}{4\delta}C_0(l+2)! ^2\varepsilon
		\\\notag
		&\qquad + C_acTC_0\frac{\gamma}{4\delta}\sum_{j=0}^{l+1}\binom{l+1}{j}j! ^2(l-j+2)! ^2\varepsilon + \frac{\gamma}{4\delta}C_0(l+3)!^2\varepsilon)
		\\\notag
		&\underbrace{\leq}_{\cref{{faculties}}}\hat cC_0\frac{\gamma}{2\delta}(C_acT + 1)(l+3)! ^2\varepsilon,
	\end{align}
	where we have used that 
	\begin{equation}\label{faculties}
		\sum_{j=0}^{l+1}\binom{l+1}{j}j! ^2(l-j+2)! ^2 =\sum_{j=0}^{l+1}\frac{(l+1)!}{(l+1-j)!j!}j! ^2(l-j+2)! ^2 = \sum_{j=0}^{l+1}(l+1)!(l-j+2)j!(l-j+2)!\leq(l+3)! ^2.
	\end{equation}
	The last inequality in \cref{faculties} is simply shown by induction. Again, with the triangular inequality and \cref{assumption4} along with \cref{ieq q check q bar2}, we obtain that 
	\begin{align*}
		\|\partial_t^l (\check q - q_0)\|_{L^2(I,L^\infty(\Omega))}
		&\leq \|\partial_t^l (\check q - \overline q)\|_{L^2(I,L^\infty(\Omega))} + \|\partial_t^l (\overline q - 
		q_0)\|_{L^2(I,L^\infty(\Omega))}
		\\
		&\leq C_0(\hat c\frac{\gamma}{2\delta}(C_acT + 1) + 1)(l+3)! ^2\varepsilon \underbrace{=}_{\cref{definition delta}} c_0 (l+3)!^2 \varepsilon
	\end{align*}
	that is \cref{ieq 1b}. By \cref{system: statetk}, $p - \tilde p$ satisfies
	\begin{equation*}
		\left\{\begin{aligned}
			&\nu_0\partial_t^2(p - \tilde p) - \Delta(p - \tilde p) + \eta\partial_t(p - \tilde p) = -(\nu - \tilde\nu)\partial_t^2p_0
			&&\text{in }I\times\Omega
			\\
			&\partial_n(p - \tilde p) = 0 
			&&\text{on }I\times\Gamma_N
			\\
			&p - \tilde p = 0 
			&&\text{on }I\times\Gamma_D	
			\\
			&(p-\tilde p, \partial_t(p - \tilde p))(0) = (0, 0)
			&&\text{in }\Omega.
		\end{aligned}\right.
	\end{equation*}
	Applying \cref{lemma: abstract}, \cref{lemma: stampaccia}, and \cref{assumption4} gives that
	\begin{equation}\label{ieq p tilde p}
		\|\partial_t^l(p-\tilde p)\|_{L^2(I,L^2(\Omega))}\leq c\sqrt{T}\|\nu - \tilde\nu\|_{L^2(\Omega)}\|\partial_t^{l+1} p_0\|_{L^1(I,L^\infty(\Omega))}\leq c C_0T(l+1)! ^2{\|\nu - \tilde\nu\|_{L^2(\Omega)}}
	\end{equation}
	and
	\begin{align*}
		\|\partial_t^l(p-\tilde p)\|_{L^2(I,L^\infty(\Omega))}
		&\leq \hat c(\|\nu - \tilde\nu\|_{L^2(\Omega)}\|\partial_t^{l+2} p_0\|_{L^2(I,L^\infty(\Omega))} + \|\nu - \tilde\nu\|_{L^2(\Omega)}\|\partial_t^{l+3} p_0\|_{L^2(I,L^\infty(\Omega))})
		\\
		&\leq 2\hat c C_0(l+3)! ^2\|\nu - \tilde\nu\|_{L^2(\Omega)},
	\end{align*}
	leading to \cref{ieq 3}; see the definition of $c_1$ in \cref{definition delta}. By \cref{system: adjointtk}, the difference $q - \tilde q$ solves 
	\begin{equation*}
		\left\{\begin{aligned}
			&\nu_0\partial_t^2(q - \tilde q) - \Delta(q - \tilde q) - \eta\partial_t(q - \tilde q) = \sum_{i=1}^m a_i(p - \tilde p) -(\nu - \tilde\nu)\partial_t^2q_0
			&&\text{in }I\times\Omega
			\\
			&\partial_n(q - \tilde q) = 0 
			&&\text{on }I\times\Gamma_N
			\\
			&q - \tilde q = 0 
			&&\text{on }I\times\Gamma_D	
			\\
			&(q-\tilde q, \partial_t(q - \tilde q))(T) = (0, 0)
			&&\text{in }\Omega,
		\end{aligned}\right.
	\end{equation*}
	such that \cref{lemma: stampaccia}, \cref{assumption4}, and \cref{ieq p tilde p} provide that
	\begin{align*}
		&\|\partial_t^l(q-\tilde q)\|_{L^2(I,L^\infty(\Omega))}
		\\
		&\leq \hat c(\sum_{j=0}^l\binom{l}{j}\sum_{i=1}^m\|\partial_t^ja_i\|_{L^\infty(I, L^\infty(\Omega))}\|\partial_t^{l-j}(p - \tilde p)\|_{L^2(I,L^2(\Omega))}+ \|\nu - \tilde\nu\|_{L^2(\Omega)}\|\partial_t^{l+2} q_0\|_{L^2(I,L^\infty(\Omega))}
		\\
		&\qquad + \sum_{j=0}^{l+1}\binom{l+1}{j}\sum_{i=1}^m\|\partial_t^ja_i\|_{L^\infty(I, L^\infty(\Omega))}\|\partial_t^{l+1-j}(p - \tilde p)\|_{L^2(I,L^2(\Omega))} + \|\nu - \tilde\nu\|_{L^2(\Omega)}\|\partial_t^{l+3} q_0\|_{L^2(I,L^\infty(\Omega))})
		\\
		&\leq \hat c(C_ac C_0T\sum_{j=0}^l\binom{l}{j}j! ^2(l-j+1)! ^2\|\nu - \tilde\nu\|_{L^2(\Omega)} + \|\nu - \tilde\nu\|_{L^2(\Omega)}C_0(l+2)! ^2
		\\
		&\qquad + C_ac C_0T\sum_{j=0}^{l+1}\binom{l+1}{j}j! ^2(l-j+2)! ^2\|\nu - \tilde\nu\|_{L^2(\Omega)} + \|\nu - \tilde\nu\|_{L^2(\Omega)}C_0(l+3)! ^2)
		\\
		&\underbrace{\leq}_{\cref{faculties}} 2\hat cC_0(C_acT + 1)(l+3)! ^2\|\nu - \tilde\nu\|_{L^2(\Omega)},
	\end{align*}
	leading to \cref{ieq 3b}; see the definition of $c_1$ in \cref{definition delta}. Now, let $k\in\mathbb N$ and we redefine $(\nu,p, q)\coloneqq T_k(\nu)$ and $(\tilde\nu,\tilde p, \tilde q)\coloneqq T_k(\tilde\nu)$. Furthermore, let $\nu_k$ and $(\nu_{k-1},p_{k-1},q_{k-1})$ satisfy \cref{assumptionk} and $p_k\in X_0, q_k\in X_T$ be the unique solutions to \cref{iterationstate} and \cref{iterationadjoint}. Due to \cref{assumption4} and \cref{assumptionk}, applying \cref{lemma: stampaccia} to \cref{iterationstate} yields that $p_k\in C^\infty(I,L^\infty(\Omega))$, $\partial_t^l p_k(0) = 0$ for all $l\in\mathbb N_0$, and
	\begin{align}\label{ieq pk}
		\|\partial_t^lp_k\|_{L^2(I,L^\infty(\Omega))}
		&\leq \hat c(\|\partial_t^lf\|_{L^2(I,L^2(\Omega))} + \|\nu_k - \nu_{k-1}\|_{L^2(\Omega)}\|\partial_t^{l+2}p_{k-1}\|_{L^2(I,L^\infty(\Omega))}
		\\\notag
		&\qquad + \|\partial_t^{l+1} f\|_{L^2(I,L^2(\Omega))} + \|\nu_k - \nu_{k-1}\|_{L^2(\Omega)}\|\partial_t^{l+3} p_{k-1}\|_{L^2(I,L^\infty(\Omega))})
		\\\notag
		&\underbrace{\leq}_{\cref{assumptionk}} 2\hat cC_f(l+1)! ^2 + \hat cC_0\frac{1}{2\delta}(b_k+b_{k-1})(l+3k)! ^2
		\\\notag
		&\underbrace{\leq}_{\cref{bk overline gamma}} 2\hat cC_f(l+1)! ^2 + \hat cC_0\frac{\overline\gamma}{\delta}(l+3k)! ^2\underbrace{\leq}_{\cref{assumption c constants}} C_0(l+3k)! ^2.
	\end{align}
	From \cref{iterationstate}, we obtain that
	\begin{equation}\label{system: overline p - pk}
		\left\{\begin{aligned}
			&\nu_{k-1}\partial_t^2 (\overline p - p_k) - \Delta(\overline p - p_k) + \eta\partial_t(\overline p - p_k) = -(\overline\nu - \nu_{k-1})\partial_t^2 \overline p + (\nu_k - \nu_{k-1})\partial_t^2 p_{k-1} 
			&&\text{in } I\times \Omega
			\\
			&\partial_n(\overline p - p_k) = 0 
			&&\text{on }I\times\Gamma_N
			\\
			&\overline p - p_k = 0 
			&&\text{on }I\times\Gamma_D	
			\\
			&(\overline p - p_k, \partial_t(\overline p - p_k))(0) = (0, 0) 
			&&\text{in } \Omega.
		\end{aligned}\right.
	\end{equation}
	Thus, \cref{lemma: abstract}, \cref{lemma: stampaccia}, \cref{assumption4}, and \cref{assumptionk} yield
	\begin{align}\notag
		\|\partial_t^l(\overline p - p_k)\|_{L^2(I,L^2(\Omega))}
		&\leq c\sqrt{T}(\|\overline\nu - \nu_{k-1}\|_{L^2(\Omega)}\|\partial_t^{l+1}\overline p\|_{L^1(I,L^\infty(\Omega))} + \|\nu_k - \nu_{k-1}\|_{L^2(\Omega)}\|\partial_t^{l+1}p_{k-1}\|_{L^1(I,L^\infty(\Omega))})
		\\\label{ieq overline p - pk}
		&\leq cT(\overline C + C_0)(l+3k-2)! ^2(\|\overline\nu - \nu_k\|_{L^2(\Omega)} + \|\overline\nu - \nu_{k-1}\|_{L^2(\Omega)})
	\end{align}
	and
	\begin{align}\label{ieq overline p - pk2}
		\|\partial_t^l (\overline p - p_k)\|_{L^2(I,L^\infty(\Omega))}
		&\leq \hat c(\|(\overline\nu - \nu_{k-1})\partial_t^{l+2} \overline p\|_{L^2(I,L^2(\Omega))} + \|(\overline\nu - \nu_{k-1})\partial_t^{l+3} \overline p\|_{L^2(I,L^2(\Omega))} 
		\\\notag
		&\qquad+ \|(\nu_k - \nu_{k-1})\partial_t^{l+2} p_{k-1}\|_{L^2(I,L^2(\Omega))} + \|(\nu_k - \nu_{k-1})\partial_t^{l+3} p_{k-1}\|_{L^2(I,L^2(\Omega))})
		\\\notag
		&\leq 2\hat c(\overline C + C_0)(l+3k)! ^2(\|\overline\nu - \nu_k\|_{L^2(\Omega)} + \|\overline\nu - \nu_{k-1}\|_{L^2(\Omega)}),
	\end{align}
	leading to \cref{ieq 7}; see the definition of $c_1$ in \cref{definition delta}. Applying \cref{lemma: stampaccia} to \cref{iterationadjoint} and along with \cref{assumption4} and \cref{ieq overline p - pk}, we obtain that 
	\begin{align*}
		&\|\partial_t^lq_k\|_{L^2(I,L^\infty(\Omega))}
		\\
		&\leq \hat c(\sum_{i=1}^{m}\|\partial_t^l(a_i(p_k - p_i^{ob}))\|_{L^2(I,L^2(\Omega))} +\|\nu_k - \nu_{k-1}\|_{L^2(\Omega)}\|\partial_t^{l+2} q_{k-1}\|_{L^2(I,L^\infty(\Omega))}
		\\
		&\hspace{.7cm} +\sum_{i=1}^{m}\|\partial_t^{l+1}(a_i(p_k - p_i^{ob}))\|_{L^2(I,L^2(\Omega))}+ \|\nu_k - \nu_{k-1}\|_{L^2(\Omega)}\|\partial_t^{l+3} q_{k-1}\|_{L^2(I,L^\infty(\Omega))})
		\\
		&\leq \hat c(\sum_{i=1}^{m}\|\partial_t^l(a_i(\overline p - p_i^{ob}))\|_{L^2(I,L^2(\Omega))} + \sum_{i=1}^{m}\|\partial_t^{l+1}(a_i(\overline p - p_i^{ob}))\|_{L^2(I,L^2(\Omega))}
		\\
		&\hspace{.7cm} + \sum_{i=1}^{m}\|\partial_t^l(a_i(p_k - \overline p))\|_{L^2(I,L^2(\Omega))} +\|\nu_k - \nu_{k-1}\|_{L^2(\Omega)}\|\partial_t^{l+2} q_{k-1}\|_{L^2(I,L^\infty(\Omega))}
		\\
		&\hspace{.7cm} +\sum_{i=1}^{m}\|\partial_t^{l+1}(a_i(p_k - \overline p))\|_{L^2(I,L^2(\Omega))}+ \|\nu_k - \nu_{k-1}\|_{L^2(\Omega)}\|\partial_t^{l+3} q_{k-1}\|_{L^2(I,L^\infty(\Omega))})
		\\
		&\leq \hat c(\sum_{i=1}^{m}\|\partial_t^l(a_i(\overline p - p_i^{ob}))\|_{L^2(I,L^2(\Omega))} + \sum_{i=1}^{m}\|\partial_t^{l+1}(a_i(\overline p - p_i^{ob}))\|_{L^2(I,L^2(\Omega))}
		\\
		&\hspace{.7cm} + \sum_{i=1}^{m}\sum_{j=0}^l\binom{l}{j}\|\partial_t^ja_i\|_{L^\infty(I, L^\infty(\Omega))}\|\partial_t^{l-j}(p_k - \overline p)\|_{L^2(I,L^2(\Omega))} +\|\nu_k - \nu_{k-1}\|_{L^2(\Omega)}\|\partial_t^{l+2} q_{k-1}\|_{L^2(I,L^\infty(\Omega))}
		\\
		&\hspace{.7cm} +\sum_{i=1}^{m}\sum_{j=0}^{l+1}\binom{l+1}{j}\|\partial_t^ja_i\|_{L^\infty(I, L^\infty(\Omega))}\|\partial_t^{l+1-j}(p_k - \overline p)\|_{L^2(I,L^2(\Omega))} + \|\nu_k - \nu_{k-1}\|_{L^2(\Omega)}\|\partial_t^{l+3} q_{k-1}\|_{L^2(I,L^\infty(\Omega))})
		\\
		&\leq \hat c(\sum_{i=1}^{m}\|\partial_t^l(a_i(\overline p - p_i^{ob}))\|_{L^2(I,L^2(\Omega))} + \sum_{i=1}^{m}\|\partial_t^{l+1}(a_i(\overline p - p_i^{ob}))\|_{L^2(I,L^2(\Omega))}
		\\
		&\hspace{.7cm} + C_acT(\overline C + C_0)\sum_{j=0}^l\binom{l}{j}j! ^2(l-j+3k-2)! ^2(\|\overline\nu - \nu_k\|_{L^2(\Omega)} + \|\overline\nu - \nu_{k-1}\|_{L^2(\Omega)})
		\\
		&\hspace{.7cm} +\|\nu_k - \nu_{k-1}\|_{L^2(\Omega)}\|\partial_t^{l+2} q_{k-1}\|_{L^2(I,L^\infty(\Omega))}
		\\
		&\hspace{.7cm} + C_acT(\overline C + C_0)\sum_{j=0}^{l+1}\binom{l+1}{j}j! ^2(l-j+3k-1)! ^2(\|\overline\nu - \nu_k\|_{L^2(\Omega)} + \|\overline\nu - \nu_{k-1}\|_{L^2(\Omega)})
		\\
		&\hspace{.7cm} + \|\nu_k - \nu_{k-1}\|_{L^2(\Omega)}\|\partial_t^{l+3} q_{k-1}\|_{L^2(I,L^\infty(\Omega))})
		\\
		&\underbrace{\leq}_{\cref{assumptionk}}\hat c(2C_a(l+1)! ^2 + C_acT(\overline C + C_0)\frac{1}{2\delta}(b_k+b_{k-1})(l+3k)! ^2 + C_0\frac{1}{2\delta}(b_k+b_{k-1})(l+3k)! ^2)
		\\
		&\underbrace{\leq}_{\cref{bk overline gamma}}\hat c(2C_a(l+1)! ^2 + C_acT(\overline C + C_0)\frac{\overline\gamma}{\delta}(l+3k)! ^2 + C_0\frac{\overline\gamma}{\delta}(l+3k)! ^2)\underbrace{\leq}_{\cref{assumption c constants}} C_0(l+3k)! ^2.
	\end{align*}
	Along with \cref{ieq pk}, the above estimate implies that \cref{pk claim} is valid. Now, \cref{ieq 5} and \cref{ieq 5b} follow analog to \cref{ieq 3} and \cref{ieq 3b}. By \cref{first order optimality} and \cref{iterationadjoint}, we obtain that
	\begin{equation*}
		\left\{\begin{aligned}
			&\nu_{k-1}\partial_t^2 (\overline q - q_k) - \Delta(\overline q - q_k) - \eta\partial_t(\overline q - q_k)=\sum_{i=1}^m a_i(\overline p - p_k) -(\overline\nu - \nu_{k-1})\partial_t^2 \overline q + (\nu_k - \nu_{k-1})\partial_t^2 q_{k-1}
			&&\text{in } I\times \Omega
			\\
			&\partial_n(\overline q - q_k) = 0 
			&&\text{on }I\times\Gamma_N
			\\
			&\overline q - q_k = 0 
			&&\text{on }I\times\Gamma_D	
			\\
			&(\overline q - q_k, \partial_t(\overline q - q_k))(T) = (0, 0) &&\text{in } \Omega,
		\end{aligned}\right.
	\end{equation*}
	such that \cref{lemma: stampaccia}, \cref{assumption4}, \cref{assumptionk}, and \cref{ieq overline p - pk} provide
	\begin{align*}
		&\|\partial_t^l (\overline q - q_k)\|_{L^2(I,L^\infty(\Omega))}
		\\
		&\leq \hat c(\sum_{j=0}^l\binom{l}{j}\sum_{i=1}^m\|\partial_t^{j}a_i\|_{L^\infty(I, L^\infty(\Omega))}\|\partial_t^{l-j}(\overline p - p_k)\|_{L^2(I,L^2(\Omega))}+ \|(\overline\nu - \nu_{k-1})\partial_t^{l+2} \overline q\|_{L^2(I,L^2(\Omega))}
		\\
		&\hspace{.7cm}+ \|(\nu_k - \nu_{k-1})\partial_t^{l+2} q_{k-1}\|_{L^2(I,L^2(\Omega))}+ \sum_{j=0}^{l+1}\binom{l+1}{j}\sum_{i=1}^m\|\partial_t^{j}a_i\|_{L^\infty(I, L^\infty(\Omega))}\|\partial_t^{l+1-j}(\overline p - p_k)\|_{L^2(I,L^2(\Omega))}
		\\
		&\hspace{.7cm} + \|(\overline\nu - \nu_{k-1})\partial_t^{l+3} \overline q\|_{L^2(I,L^2(\Omega))} + \|(\nu_k - \nu_{k-1})\partial_t^{l+3} q_{k-1}\|_{L^2(I,L^2(\Omega))})
		\\
		&\leq \hat c(C_acT(\overline C + C_0)\sum_{j=0}^l\binom{l}{j}j! ^2(l+3k-j-2)! ^2+ \overline C(l+2)! ^2 + C_0(l+3k-1)! ^2
		\\
		&\hspace{.7cm}+ C_acT(\overline C + C_0)\sum_{j=0}^{l+1}\binom{l+1}{j}j! ^2(l+3k - j - 1)! ^2+ \overline C(l+3)! ^2
		\\
		&\qquad+ C_0(l+3k)! ^2)(\|\overline\nu - \nu_k\|_{L^2(\Omega)} + \|\overline\nu - \nu_{k-1}\|_{L^2(\Omega)})
		\\
		&\leq \hat c(2C_acT(\overline C + C_0) + 2\overline C + 2C_0)(l+3k)! ^2(\|\overline\nu - \nu_k\|_{L^2(\Omega)} + \|\overline\nu - \nu_{k-1}\|_{L^2(\Omega)}),
	\end{align*}
	leading to \cref{ieq 8}; see the definition of $c_1$ in \cref{definition delta}. To prove \cref{ieq 9}, due to \cref{system: statetk} and \cref{iterationstate}, note that $p - p_k$ solves
	\begin{equation}\label{system: p - p_k}
		\left\{\begin{aligned}
			&\nu_k\partial_t^2 (p - p_k) - \Delta(p - p_k) + \eta\partial_t(p - p_k) = - (\nu - \nu_{k-1})\partial_t^2 p_k + (\nu_k - \nu_{k-1})\partial_t^2 p_{k-1}
			&&\text{in } I\times \Omega
			\\
			&\partial_n(p - p_k) = 0 
			&&\text{on }I\times\Gamma_N
			\\
			&p - p_k = 0 
			&&\text{on }I\times\Gamma_D	
			\\
			&(p - p_k, \partial_t (p - p_k))(0) = (0, 0) 
			&&\text{in } \Omega,
		\end{aligned}\right.
	\end{equation}
	such that from \cref{lemma: abstract}, \cref{lemma: stampaccia}, \cref{pk claim}, and \cref{assumptionk}, it follows that
	\begin{align}\notag
		\|\partial_t^l(p - p_k)\|_{L^2(I,L^2(\Omega))}
		&\leq c\sqrt{T}(\|\nu - \nu_{k-1}\|_{L^2(\Omega)}\|\partial_t^{l+1}p_k\|_{L^1(I,L^\infty(\Omega))} + \|\nu_k - \nu_{k-1}\|_{L^2(\Omega)}\|\partial_t^{l+1}p_{k-1}\|_{L^1(I,L^\infty(\Omega))})
		\\\label{ieq p - pk3}
		&\leq cTC_0(l+3k+1)! ^2(\|\nu - \nu_k\|_{L^2(\Omega)} + \|\nu - \nu_{k-1}\|_{L^2(\Omega)})
	\end{align}
	and
	\begin{align}\notag
		\|\partial_t^l(p - p_k)\|_{L^2(I,L^\infty(\Omega))}
		&\leq \hat c(\|\nu - \nu_{k-1}\|_{L^2(\Omega)}\|\partial_t^{l+2}p_k\|_{L^2(I,L^\infty(\Omega))} + \|\nu_k - \nu_{k-1}\|_{L^2(\Omega)}\|\partial_t^{l+2}p_{k-1}\|_{L^2(I,L^\infty(\Omega))}
		\\\notag
		&\qquad + \|\nu - \nu_{k-1}\|_{L^2(\Omega)}\|\partial_t^{l+3} p_k\|_{L^2(I,L^\infty(\Omega))} + \|\nu_k - \nu_{k-1}\|_{L^2(\Omega)}\|\partial_t^{l+3} p_{k-1}\|_{L^2(I,L^\infty(\Omega))})
		\\\label{ieq p - pk2}
		&\leq 4\hat cC_0(l+3k+3)! ^2(\|\nu - \nu_k\|_{L^2(\Omega)} + \|\nu - \nu_{k-1}\|_{L^2(\Omega)}),
	\end{align}
	leading to \cref{ieq 9}; see the definition of $c_1$ in \cref{definition delta}. Finally, due to \cref{system: adjointtk} and \cref{iterationadjoint}, it holds that
	\begin{equation*}
		\left\{\begin{aligned}
			&\nu_k\partial_t^2 (q - q_k) - \Delta(q - q_k) - \eta\partial_t(q - q_k) = \sum_{i=1}^m a_i(p - p_k) - (\nu - \nu_{k-1})\partial_t^2 q_k + (\nu_k - \nu_{k-1})\partial_t^2 q_{k-1}
			&&\text{in } I\times \Omega
			\\
			&\partial_n(q - q_k) = 0 
			&&\text{on }I\times\Gamma_N
			\\
			&q - q_k = 0 
			&&\text{on }I\times\Gamma_D	
			\\
			&(q - q_k, \partial_t (q - q_k))(T) = (0, 0) 
			&&\text{in } \Omega,
		\end{aligned}\right.
	\end{equation*}
	such that \cref{lemma: stampaccia}, \cref{assumption4}, \cref{assumptionk}, \cref{pk claim}, and \cref{ieq p - pk3} yield that
	\begin{align*}
		&\|\partial_t^l(q - q_k)\|_{L^2(I,L^\infty(\Omega))}
		\\
		&\leq \hat c(\sum_{j=0}^l\binom{l}{j}\sum_{i=1}^{m}\|\partial_t^ja_i\|_{L^\infty(I, L^\infty(\Omega))}\|\partial_t^{l-j}(p - p_k)\|_{L^2(I,L^2(\Omega))} + \|\nu - \nu_{k-1}\|_{L^2(\Omega)}\|\partial_t^{l+2}q_k\|_{L^2(I,L^\infty(\Omega))}
		\\
		&\hspace{.7cm} + \|\nu_k - \nu_{k-1}\|_{L^2(\Omega)}\|\partial_t^{l+2}q_{k-1}\|_{L^2(I,L^\infty(\Omega))} + \sum_{j=0}^{l+1}\binom{l+1}{j}\sum_{i=1}^{m}\|\partial_t^ja_i\|_{L^\infty(I, L^\infty(\Omega))}\|\partial_t^{l+1-j}(p - p_k)\|_{L^2(I,L^2(\Omega))} 
		\\
		&\hspace{.7cm} + \|\nu - \nu_{k-1}\|_{L^2(\Omega)}\|\partial_t^{l+3} q_k\|_{L^2(I,L^\infty(\Omega))} + \|\nu_k - \nu_{k-1}\|_{L^2(\Omega)}\|\partial_t^{l+3} q_{k-1}\|_{L^2(I,L^\infty(\Omega))})
		\\
		&\leq \hat c(C_acTC_0\sum_{j=0}^l\binom{l}{j}j! ^2(l-j+3k+1)! ^2 + 4C_0(l+3k+3)! ^2
		\\
		&\qquad+ C_acTC_0\sum_{j=0}^{l+1}\binom{l+1}{j}j! ^2(l-j+3k+2)! ^2)(\|\nu - \nu_k\|_{L^2(\Omega)} + \|\nu - \nu_{k-1}\|_{L^2(\Omega)})
		\\
		&\leq \hat c(2C_acTC_0 + 4C_0)(l+3k+3)! ^2(\|\nu - \nu_k\|_{L^2(\Omega)} + \|\nu - \nu_{k-1}\|_{L^2(\Omega)}),
	\end{align*}
	leading to \cref{ieq 10}; see the definition of $c_1$ in \cref{definition delta}.
\end{proof}

\bibliographystyle{abbrvurl}
\footnotesize

\bibliography{references}

\end{document}